\documentclass{amsart}
\usepackage{amscd}
\usepackage{amsmath,empheq}
\usepackage{amsfonts}
\usepackage{amssymb}
\usepackage{mathrsfs}

\usepackage[all]{xy}

\newtheorem{theorem}{Theorem}

\newtheorem{lemma}[theorem]{Lemma}
\newtheorem{prop}[theorem]{Proposition}
\newtheorem{remark}{Remark}
\newtheorem{corollary}[theorem]{Corollary}
\newtheorem{claim}{Claim}

\newenvironment{proof-sketch}{\noindent{\bf Sketch of Proof}\hspace*{1em}}{\qed\bigskip}

\everymath{\displaystyle}

\newcommand{\RR}{\mathbb R}
\newcommand{\NN}{\mathbb N}

\newcommand{\di}{\displaystyle}

\renewcommand{\leq}{\leqslant}

\renewcommand{\geq}{\geqslant}

\begin{document}

\title[Perturbations of nonlinear eigenvalue problems]{Perturbations of nonlinear eigenvalue problems}

\author[N.S. Papageorgiou]{Nikolaos S. Papageorgiou}
\address[N.S. Papageorgiou]{National Technical University, Department of Mathematics,
				Zografou Campus, Athens 15780, Greece \& Institute of Mathematics, Physics and Mechanics, Jadranska 19, 1000 Ljubljana, Slovenia}
\email{\tt npapg@math.ntua.gr}

\author[V.D. R\u{a}dulescu]{Vicen\c{t}iu D. R\u{a}dulescu}
\address[V.D. R\u{a}dulescu]{Institute of Mathematics, Physics and Mechanics, Jadranska 19, 1000 Ljubljana, Slovenia \& Faculty of Applied Mathematics, AGH University of Science and Technology, al. Mickiewicza 30, 30-059 Krak\'ow, Poland \& Institute of Mathematics ``Simion Stoilow" of the Romanian Academy, P.O. Box 1-764,
          014700 Bucharest, Romania}
\email{\tt vicentiu.radulescu@imar.ro}

\author[D.D. Repov\v{s}]{Du\v{s}an D. Repov\v{s}}
\address[D.D. Repov\v{s}]{Institute of Mathematics, Physics and Mechanics, Jadranska 19, 1000 Ljubljana, Slovenia \& Faculty of Education and Faculty of Mathematics and Physics, University of Ljubljana, 1000 Ljubljana, Slovenia}
\email{\tt dusan.repovs@guest.arnes.si}


\begin{abstract}
We consider perturbations of nonlinear eigenvalue problems driven by a nonhomogeneous differential operator plus an indefinite potential. We consider both sublinear and superlinear perturbations and we determine how the set of positive solutions changes as the real parameter $\lambda$ varies. We also show that there exists a minimal positive solution $\overline{u}_\lambda$ and determine the monotonicity and continuity properties of the map $\lambda\mapsto\overline{u}_\lambda$. Special attention is given to the special case of the $p$-Laplacian.
\end{abstract}

\maketitle

\section{Introduction}

The aim of this paper is to study the following nonlinear nonhomogeneous parametric Robin problem
\begin{equation}
	\left\{
		\begin{array}{ll}
			-{\rm div}\,a(Du(z)) + \xi(z)u(z)^{p-1} = \lambda u(z)^{p-1} + f(z,u(z))\ \mbox{in}\ \Omega, \\
			\frac{\partial u}{\partial n_a} + \beta(z)u^{p-1}=0\ \mbox{on}\ \partial\Omega,\ u>0, \ \lambda\in\RR,\ 1<p<\infty.
		\end{array}
	\right\}
	\tag{$P_\lambda$}\label{eqp}
\end{equation}

In this problem, $\Omega\subseteq\RR^N$ is a bounded domain with a $C^2$-boundary $\partial\Omega$. The map $a:\RR^N\rightarrow\RR^N$ involved in the differential operator, is continuous, strictly monotone (hence maximal monotone, too) and satisfies some other regularity and growth conditions listed in hypotheses $H(a)$ below (see Section 2). These extra-conditions on $a(\cdot)$ are not restrictive and so our framework incorporates many differential operators of interest such as the $p$-Laplacian and the $(p,q)$-Laplacian (that is, the sum of a $p$-Laplacian and a $q$-Laplacian). The potential function $\xi\in L^\infty(\Omega)$ is indefinite (that is, sign changing). In the reaction (the right-hand side of the equation), we have a parametric term $u\mapsto\lambda u^{p-1}$ and a perturbation $f(z,x)$ which is a Carath\'eodory function (that is, for all $x\in\RR$ the mapping $z\mapsto f(z,x)$ is measurable and for almost all $z\in\Omega$ the mapping $x\mapsto f(z,x)$ is continuous).

We consider two distinct cases. In the first one, $f(z,\cdot)$ is $(p-1)$-sublinear near $+\infty$, while in the second one we assume that $f(z,\cdot)$ is $(p-1)$-superlinear. In the boundary condition, $\frac{\partial u}{\partial n_a}$ denotes the conormal derivative of $u$, defined by extension of the map
$$
L^1(\overline{\Omega})\ni u\mapsto (a(Du),n)_{\RR^N},
$$
with $n(\cdot)$ being the outward unit normal on $\partial\Omega$. The boundary coefficient $\beta(\cdot)$ is non-negative and the case $\beta\equiv0$ is also included and corresponds to the Neumann problem.

We look for positive solutions of \eqref{eqp} and we want to determine how the set of positive solutions changes as the parameter $\lambda$ moves along
 the real line $\RR$. More precisely, we show that there is a critical parameter value $\lambda^*\in\RR$ such that for $\lambda<\lambda^*$ problem \eqref{eqp} has
\begin{itemize}
	\item at least one positive smooth solution, when $f(z,\cdot)$ is $(p-1)$-sublinear;
	\item at least two positive smooth solutions, when $f(z,\cdot)$ is $(p-1)$-superlinear.
\end{itemize}

For $\lambda\geq\lambda^*$, problem \eqref{eqp} has no positive solutions.

In the special case of the $p$-Laplace differential operator (that is, $a(y)=|y|^{p-2}y$ for all $y\in\RR^N$),  problem \eqref{eqp} can be viewed as a perturbation of the classical eigenvalue problem for the $p$-Laplacian. In this particular case, we can identify $\lambda^*$ as the principal eigenvalue $\hat\lambda_1$ of the differential operator $u\mapsto -\Delta_pu+\xi(z)|u|^{p-2}u$ with the Robin boundary condition. This was already observed by these authors for the semilinear problem (that is, $p=2$ hence $a(y)=y$ for all $y\in\RR^N$), see Papageorgiou, R\u{a}dulescu \& Repov\v{s} \cite{21}. Also, for both cases (sublinear and superlinear), we establish the existence of a smallest positive solution $\overline{u}_\lambda$ and determine the monotonicity and continuity properties of the map $\lambda\mapsto\overline{u}_\lambda$. Finally, in the sublinear case we address the question of uniqueness of the  solution.

Nonlinear nonhomogeneous parametric Robin problems were also studied by Autuori \& Pucci \cite{2}, Colasuonno, Pucci \& Varga \cite{5}, Fragnelli, Mugnai \& Papageorgiou \cite{7}, Papageorgiou, R\u{a}dulescu \& Repov\v{s} \cite{22, 22bis}, and Perera, Pucci \& Varga \cite{23}.

Our approach is variational, using results from the critical point theory and also truncation, perturbation and comparison techniques.

\section{Mathematical background and hypotheses}

Let $X$ be a Banach space and $X^*$ its topological dual. We denote by $\langle\cdot,\cdot\rangle$  the duality brackets for the pair $(X^*,X)$. Given $\varphi\in C^1(X,\RR)$, we say that $\varphi$ satisfies the ``Cerami condition" (the ``C-condition" for short), if the following property holds:
\begin{center}
``Every sequence $\{u_n\}_{n\geq 1}\subseteq X$ such that
$$\{\varphi(u_n)\}_{n\geq 1}\subseteq\RR\ \mbox{is bounded and}\
\mbox{$(1+||u_n||)\varphi'(u_n)\rightarrow 0$ in $X^*$ as $n\rightarrow\infty$,}$$
admits a strongly convergent subsequence."
\end{center}

In what follows, we denote  by $K_\varphi$ the critical set of $\varphi$, that is,
$$
K_\varphi=\{u\in X:\varphi'(u)=0\}.
$$

Also, if $c\in\RR$, then
$$
K^c_\varphi=\{u\in K_\varphi:\varphi(u)=c\}.
$$

Using the notion of the C-condition, we  have the following minimax theorem, known in the literature as the ``mountain pass theorem".

\begin{theorem}\label{th1}
If $\varphi\in C^1(X,\RR)$ satisfies the C-condition, $u_0,u_1\in X$, $||u_1-u_0||>\rho>0$,
	$$\max\{\varphi(u_0),\varphi(u_1)\}<\inf\{\varphi(u):||u-u_0||=\rho\}=m_{\rho}$$
	and $c=\inf\limits_{\gamma\in\Gamma}\max\limits_{0\leq t\leq 1}\ \varphi(\gamma(t))$ with $\Gamma=\{\gamma\in C([0,1],X):\gamma(0)=u_0,\gamma(1)=u_1\}$, then $c\geq m_{\rho}$ and $c$ is a critical value of $\varphi$ (that is, we can find $\hat{u}\in X$ such that $\varphi'(\hat{u})=0$ and $\varphi(\hat{u})=c$).
\end{theorem}

In the analysis of problem \eqref{eqp}, we will use the Sobolev space $W^{1,p}(\Omega)$, the Banach space $C^1(\overline{\Omega})$ and the ``boundary" Lebesgue spaces $L^r(\partial\Omega), 1\leq r\leq\infty$. We denote by $||\cdot||$ the norm of $W^{1,p}(\Omega)$ defined by
$$
||u||=\left(||u||^p_p + ||Du||^p_p\right)^{1/p}\ \mbox{for all}\ u\in W^{1,p}(\Omega).
$$

The Banach space $C^1(\overline{\Omega})$ is an ordered Banach space with positive (order) cone $C_+=\{u\in C^1(\overline{\Omega}):u(z)\geq0\ \mbox{for all}\ z\in\overline{\Omega}\}$. This cone has a nonempty interior given by
$$
{\rm int}\,C_+=\left\{u\in C_+:u(z)>0\ \mbox{for all}\ z\in\Omega, \left.\frac{\partial u}{\partial n}\right|_{\partial\Omega\cap u^{-1}(0)}<0\right\}.
$$

Evidently, ${\rm int}\,C_+$ contains the open set $D_+$ defined by
$$
D_+=\{u\in C_+:u(z)>0\ \mbox{for all}\ z\in\overline{\Omega}\}.
$$

In fact, $D_+$ is the interior of $C_+$ when $C^1(\overline{\Omega})$ is furnished with the $C(\overline{\Omega})$-norm topology.

On $\partial\Omega$ we consider the $(N-1)$-dimensional Hausdorff (surface) measure $\sigma(\cdot)$. Using this measure on $\partial\Omega$, we can define in the usual way the ``boundary" Lebesgue spaces $L^r(\partial\Omega)$. From the theory of Sobolev spaces, we know that there exists a unique continuous linear map $\gamma_0:W^{1,p}(\Omega)\rightarrow L^p(\partial\Omega)$, known as the ``trace map", such that
$$
\gamma_0(u)=u|_{\partial\Omega}\ \mbox{for all}\ u\in W^{1,p}(\Omega)\cap C(\overline{\Omega}).
$$

So, the trace map extends the notion of boundary value to all Sobolev functions. The map $\gamma_0(\cdot)$ is compact into $L^r(\partial\Omega)$ for all $r\in\left[1,(N-1)p/(N-p)\right)$ when $p<N$, and into $L^r(\partial\Omega)$ for all $1\leq r<\infty$ when $p\geq\NN$. Also, we have
$$
{\rm im}\,\gamma_0=W^{\frac{1}{p},p}(\partial\Omega)\left(\frac{1}{p}+\frac{1}{p'}=1\right)\ \mbox{and}\ {\rm ker}\,\gamma_0=W^{1,p}_0(\Omega).
$$

In what follows, for the sake of notational simplicity, we drop the use of the trace map $\gamma_0(\cdot)$. All restrictions of Sobolev functions on $\partial\Omega$, are understood in the sense of traces.

Let $\vartheta\in C^1(0,\infty)$ and assume that it satisfies the following growth conditions:
\begin{equation}\label{eq1}
	0<\hat{c}\leq\frac{\vartheta'(t)t}{\vartheta(t)}\leq c_0\ \mbox{and}\ c_1t^{p-1}\leq\vartheta(t)\leq c_2(t^{\tau-1}+t^{p-1})\ \mbox{for all}\ t>0,
\end{equation}
with $0<c_1,c_2$ and $1\leq\tau<p$.

The hypotheses on the map $a(\cdot)$ involved in the differential operator of \eqref{eqp}, are as follows:

\smallskip
$H(a):$ $a(y)=a_0(|y|)y$ for all $y\in\RR^N$ with $a_0(t)>0$ for all $t>0$ and
\begin{itemize}
	\item [(i)] $a_0\in C^1(0,\infty)$, $t\mapsto a_0(t)t$ is strictly increasing on $(0,\infty), a_0(t)t\rightarrow0^+$ as $t\rightarrow0^+$ and $\lim_{t\rightarrow0^+}\frac{a'_0(t)t}{a_0(t)}>-1$;
	\item [(ii)] $|\nabla a(y)|\leq c_3\frac{\vartheta(|y|)}{|y|}$ for all $y\in\RR^N\backslash\{0\}$, and for some $c_3>0$;
	\item [(iii)] $(\nabla a(y)\xi,\xi)_{\RR^N}\geq\frac{\vartheta(|y|)}{|y|}|\xi|^2$ for all $y\in\RR^N\backslash\{0\}$, $\xi\in\RR^N$;
	\item [(iv)] if $G_0(t)=\int^t_0a_0(s)sds$ then there exists $q\in(1,p]$ such that $t\mapsto G_0(t^{{1}/{q}})$ is convex, $\lim_{t\rightarrow0^+}\frac{qG_0(t)}{t^q}=\tilde{c}>0$ and $0\leq pG_0(t)-a_0(t)t^2$ for all $t\geq0$.
\end{itemize}

\begin{remark}
	Hypotheses $H(a)(i)(ii)(iii)$ permit the use of the nonlinear regularity theory of Lieberman \cite{12} and of the nonlinear maximum principle of Pucci \& Serrin \cite{24}. Hypothesis $H(a)(iv)$ serves the needs of our problem. It is a mild condition which is satisfied in all cases of interest (see the examples below). These hypotheses imply that $G_0(\cdot)$ is strictly increasing and strictly convex. We set $G(y)=G_0(|y|)$ for all $y\in\RR^N$. We have
	$$
	\nabla G(y)=G_0'(|y|)\frac{y}{|y|}=a_0(|y|)y=a(y)\ \mbox{for all}\ y\in\RR^N, \nabla G(0)=0.
	$$

So, $G_0(\cdot)$ is the primitive of $a(\cdot)$ and $G(\cdot)$ is convex, $G(0)=0$. Therefore
\begin{equation}\label{eq2}
	G(y)\leq(a(y),y)_{\RR^N}\ \mbox{for all}\ y\in\RR^N.
\end{equation}
\end{remark}

The next lemma summarizes the main properties of the map $a(\cdot)$. It is an easy consequence of hypotheses $H(a)(i),\,(ii),\,(iii)$.

\begin{lemma}\label{lem2}
	If hypotheses $H(a)(i)(ii)(iii)$ hold, then
	\begin{itemize}
		\item [(a)] $y\mapsto a(y)$ is continuous and strictly monotone (thus
		 maximal monotone, too);
		\item [(b)] $|a(y)|\leq c_4\left(1+|y|^{p-1}\right)$ for all $y\in\RR^N$, and for some $c_4>0$;
		\item [(c)] $(a(y),y)_{\RR^N}\geq\frac{c_1}{p-1}|y|^p$ for all $y\in\RR^N$.
	\end{itemize}
\end{lemma}

Using this lemma and relation (\ref{eq2}), we obtain the following growth properties for the primitive $G(\cdot)$

\begin{corollary}\label{cor3}
	If hypotheses $H(a)(i),\,(ii),\,(iii)$ hold, then $\frac{c_1}{p(p-1)}|y|^p\leq G(y)\leq c_5(1+|y|^p)$ for all $y\in\RR^N$, and for some $c_5>0$.
\end{corollary}

{\it Examples}. The following maps $a(\cdot)$ satisfy hypotheses $H(a)$ (see also Papageorgiou \& R\u{a}dulescu \cite{18, 19}):
\begin{itemize}
	\item [(a)] $a(y)=|y|^{p-2}y,\ 1<p<\infty$. \\
		This map corresponds to the $p$-Laplace differential operator defined by
		$$
		\Delta_p u={\rm div}\,\left[|Du|^{p-2}Du\right]\ \mbox{for all}\ u\in W^{1,p}(\Omega).
		$$
	\item [(b)] $a(y)=|y|^{p-2}y+|y|^{q-2}y,\ 1<q<p<\infty$.\\
		This map corresponds to the $(p,q)$-Laplace differential operator defined by
		$$
		\Delta_p u+\Delta_q u\ \mbox{for all}\ u\in W^{1,p}(\Omega).
		$$
		Such operators arise in problems of mathematical physics (see Cherfils \& Ilyasov \cite{4}). A survey of some recent results on such equations with several relevant references, can be found in Marano \& Mosconi \cite{14}.
	\item [(c)] $a(y)=[1+|y|^2]^{\frac{p-2}{2}}y,\ 1<p<\infty$.\\
		This map corresponds to the generalized $p$-mean curvature differential operator defined by
		$$
		{\rm div}\,[(1+|Du|^2)^{\frac{p-2}{2}}Du]\ \mbox{for all}\ u\in W^{1,p}(\Omega).
		$$
	\item [(d)] $a(y):|y|^{p-2}y\left(1+\frac{1}{1+|y|^p}\right),\ 1<p<\infty$
\end{itemize}

Let $A:W^{1,p}(\Omega)\rightarrow W^{1,p}(\Omega)^*$ be the nonlinear map defined by
$$
\langle A(u),h\rangle=\int_\Omega(a(Du),Dh)_{\RR^N}dz\ \mbox{for all}\ u,h\in W^{1,p}(\Omega).
$$

The next proposition establishes the properties of $A(\cdot)$ and is a special case of a more general result of Gasinski \& Papageorgiou \cite{9} (see also Gasinski \& Papageorgiou \cite[Problem 2.192]{10}).

\begin{prop}\label{prop4}
	If hypotheses $H(a)$ hold, then $A(\cdot)$ is bounded (that is, maps bounded sets to bounded sets), continuous, monotone (thus maximal monotone, too) and of type $(S)_+$, that is, if $u_n\xrightarrow{w}u$ in $W^{1,p}(\Omega)$ and $\limsup_{n\rightarrow\infty}\langle A(u_n),u_n-u\rangle\leq0$, then $u_n\rightarrow u$ in $W^{1,p}(\Omega)$.
\end{prop}

We will also need the following strong comparison principle due to Papageorgiou, R\u{a}dulescu \& Repov\v{s} \cite{21}.

\begin{prop}\label{prop5}
	If hypotheses $H(a)$ hold, $k\in L^\infty(\Omega)$ with $k(z)\geq0$ for almost all $z\in\Omega$, $h_1, h_2\in L^\infty(\Omega)$
	$$
	0<\tilde{\gamma}\leq h_2(z)-h_1(z)\ \mbox{for almost all}\ z\in\Omega
	$$
	and $u,v\in C^{1,\alpha}(\overline\Omega)$ with $\alpha\in(0,1], u\leq v$ and
	\begin{align*}
		-{\rm div}\,a(Du) + k(z)|u|^{p-2}u=h_1(z)\ \mbox{for almost all}\ z\in\Omega, \\
		-{\rm div}\,a(Dv) + k(z)|v|^{p-2}v=h_2(z)\ \mbox{for almost all}\ z\in\Omega,
	\end{align*}
	then $v-u\in {\rm int}\,C_+$.
\end{prop}

We introduce the following hypotheses on the potential function $\xi(\cdot)$ and the boundary coefficient $\beta(\cdot)$

\smallskip
${H(\xi)}:$ $\xi\in L^\infty(\Omega)$;

\smallskip
${H(\beta)}:$ $\beta\in C^{0,\alpha}(\partial\Omega)$ with $\alpha\in(0,1)$ and $\beta(z)\geq0$ for all $z\in\partial\Omega$.

\begin{remark}
	The case $\beta\equiv0$ corresponds to the Neumann problem.
\end{remark}

Let $\mu: W^{1,p}(\Omega)\rightarrow\RR$ be the $C^1$-functional defined by
$$
\mu(u)=\int_\Omega pG(Du)dz + \int_\Omega\xi(z)|u|^pdz + \int_{\partial\Omega}\beta(z)|u|^pd\sigma\ \mbox{for all}\ u\in W^{1,p}(\Omega).
$$

Consider a Carath\'eodory function $f_0:\Omega\times\RR\rightarrow\RR$ satisfying
$$
|f_0(z,x)|\leq a_0(z)(1+|x|^{r-1})\ \mbox{for almost all}\ z\in\Omega,\ \mbox{and for all}\ x\in\RR,
$$
with $a_0\in L^{\infty}(\Omega)$ and $1\leq r\leq p^*=\left\{\begin{array}{ll}\frac{Np}{N-p} & \mbox{if}\ p<N\\ +\infty & \mbox{if}\ N\leq p\end{array}\right.$ (the critical Sobolev exponent).

We set $F_0(z,x)=\int_0^xf_0(z,s)ds$ and consider the $C^1$-functional $\varphi_0:W^{1,p}(\Omega)\rightarrow\RR$ defined by
$$
\varphi_0(v)=\frac{1}{p}\mu(u) - \int_\Omega F_0(z,u)dz\ \mbox{for all}\ u\in W^{1,p}(\Omega).
$$

From Papageorgiou \& R\u{a}dulescu \cite{20} we have the following proposition. The result is essentially an outgrowth of the nonlinear regularity theory of Lieberman \cite{12}.

\begin{prop}\label{prop6}
	Assume that hypotheses $H(a)$ hold and $u_0\in W^{1,p}(\Omega)$ is a local $C^1(\overline\Omega)$-minimizer of $\varphi_0(\cdot)$, that is, there exists $\rho_1>0$ such that
	$$
	\varphi_0(u_0)\leq\varphi_0(u_0+h)\ \mbox{for all}\ h\in C^1(\overline\Omega)\ \mbox{with}\ ||h||_{C^1(\overline\Omega)}\leq\rho_1.
	$$
	Then $u_0\in C^{1,\alpha}(\overline\Omega)$ for some $\alpha\in(0,1)$ and $u_0$ is a local $W^{1,p}(\Omega)$-minimizer of $\varphi_0$, that is, there exists $\rho_2>0$ such that
	$$
	\varphi_0(u_0)\leq\varphi_0(u_0+h)\ \mbox{for all}\ h\in W^{1,p}(\Omega)\ \mbox{with}\ ||h||\leq\rho_2.
	$$
\end{prop}

We will also use some facts about the spectrum of the following nonlinear eigenvalue problem:
\begin{equation}\label{eq3}
	\left\{
		\begin{array}{ll}
			-\Delta_ru(z) + \xi(z)|u(z)|^{r-2}u(z) = \hat\lambda|u(z)|^{r-2}u(z)\ \mbox{in}\ \Omega, \\
			\frac{\partial u}{\partial n_r} + \beta(z)|u|^{r-2} u = 0\ \mbox{on}\ \partial\Omega, 1<r<\infty.
		\end{array}
	\right\}
\end{equation}

In this case, the conormal derivative $\frac{\partial u}{\partial n_r}$ is defined by
$$
\frac{\partial u}{\partial n_r} = |Du|^{r-2}\frac{\partial u}{\partial n}\ \mbox{for all}\ u\in W^{1,r}(\Omega).
$$

As before, $n(\cdot)$ denotes the outward unit normal on $\partial\Omega$. We say that $\hat\lambda\in\RR$ is an ``eigenvalue", if problem (\ref{eq3}) admits a nontrivial solution $\hat{u}\in W^{1,r}(\Omega)$, known as an ``eigenfunction" corresponding to the eigenvalue $\hat\lambda$. The nonlinear regularity theory of Lieberman \cite{12} (see also Gasinski \& Papageorgiou \cite[pp. 737-738]{8}), implies that $\hat{u}\in C^1(\overline\Omega)$. From Fragnelli, Mugnai \& Papageorgiou \cite{7} (see also Mugnai \& Papageorgiou \cite{16} and Papageorgiou \& R\u{a}dulescu \cite{17}, where special cases of (\ref{eq3}) are discussed), we have the following property.

\begin{prop}\label{prop7}
	If hypotheses $H(\xi), H(\beta)$ hold, then problem (\ref{eq3}) admits a smallest eigenvalue $\hat{\lambda_1}=\hat{\lambda_1}(r,\xi,\beta)\in\RR$ such that
	\begin{itemize}
		\item [(a)] $\hat{\lambda_1}$ is isolated (that is, if $\hat{\sigma}(r)$ denotes the spectrum of (\ref{eq3}), then we can find $\epsilon>0$ such that $(\hat\lambda_1,\hat\lambda_1+\epsilon)\cap\hat{\sigma}(r)=\emptyset$);
		\item [(b)] $\hat\lambda_1$ is simple (that is, if $\hat{u},\hat{v}\in C^1(\overline\Omega)$ are eigenfunctions corresponding to $\hat\lambda_1$, then $\hat{u}=\hat{\xi}\hat{v}$ for some $\hat{\xi}\in\RR\backslash\{0\}$);
		\item [(c)] we have
		\begin{equation}\label{eq4}
			\hat\lambda_1=\inf\left\{\frac{\mu_r(u)}{||u||^r_r}:u\in W^{1,r}(\Omega), u\neq0\right\},
		\end{equation}
with
$$\mu_r(u)=\|Du\|^r_r+\int_\Omega \xi(z)|u|^rdz+\int_{\partial\Omega}\beta(z)|u|^rd\sigma.$$
	\end{itemize}
\end{prop}

In (\ref{eq4}), the infimum is realized on the corresponding one-dimensional eigenspace. The above properties imply that the elements of this eigenspace have fixed sign. We denote by $\hat{u}_1=\hat{u}_1(r,\xi,\beta)$ the positive, $L^r$-normalized (that is, $||\hat{u}_1||_r=1$) eigenfunction corresponding to $\hat\lambda_1=\hat\lambda(r,\xi,\beta)$. The nonlinear Hopf lemma (see Pucci \& Serrin \cite[pp. 111, 120]{24} and Gasinski \& Papageorgiou \cite[p. 738]{8}) implies that $\hat{u}_1\in D_+$. Moreover, if $\hat{u}$ is an eigenfunction corresponding to an eigenvalue $\hat{\lambda}\neq\hat{\lambda}_1$, then $\hat{u}$ is nodal (that is, sign-changing).

For every $x\in\RR$, let $x^\pm=\max\{\pm x,0\}$. Then given $u\in W^{1,p}(\Omega)$, we set $u^\pm(\cdot)=u(\cdot)^\pm$. We know that
$$
u^\pm\in W^{1,p}(\Omega),\ |u|=u^++u^-,\ u=u^+-u^-.
$$

Given a measurable function $k:\Omega\times\RR\rightarrow\RR$ (for example, a Carath\'eodory function), we denote by $N_k(\cdot)$  the Nemytskii map corresponding to $k(\cdot,\cdot)$, that is,
$$
N_k(u)(\cdot) = k(\cdot,u(\cdot))\ \mbox{for all}\ u\in W^{1,p}(\Omega).
$$

If $v,u\in W^{1,p}(\Omega)$ and $v\leq u$, then we set
$$
	\begin{array}{ll}
		& [v,u] = \{y\in W^{1,p}(\Omega):v(z)\leq y(z)\leq u(z)\ \mbox{for almost all}\ z\in\Omega\} \\
	 & [u)=\{y\in W^{1,p}(\Omega): u(z)\leq y(z)\ \mbox{for almost all}\ z\in\Omega\}.
	\end{array}
$$

\section{($p-1$)-sublinear perturbation}

In this section, we examine the case where the perturbation $f(z,x)$ in problem \eqref{eqp} is $(p-1)$-sublinear near $+\infty$. More precisely, the hypotheses on $f(z,x)$ are the following:

\smallskip
${H(f)_1}$: $f:\Omega\times\RR\rightarrow\RR$ is a Carath\'eodory function such that $f(z,0)=0$ for almost all $z\in\Omega$ and
\begin{itemize}
	\item [(i)] for every $\rho>0$, there exists $a_\rho\in L^\infty(\Omega)$ such that
		$$
		|f(z,x)|\leq a_\rho(z)\ \mbox{for almost all}\ z\in\Omega,\ \mbox{and for all}\ 0\leq x\leq\rho;
		$$
	\item [(ii)] $\lim_{x\rightarrow+\infty}\frac{f(z,x)}{x^{p-1}}=0$ uniformly for almost all $z\in\Omega$;
	\item [(iii)] with $q\in(1,p]$ as in hypothesis $H(a)(iv)$ we have
		$$
		\lim_{x\rightarrow0^+}\frac{f(z,x)}{x^{q-1}} = +\infty\ \mbox{uniformly for almost all}\ z\in\Omega;
		$$
	\item [(iv)] for every $\rho>0$, there exists $\hat{\xi}_\rho>0$ such that for almost all $z\in\Omega$, the
	 function
		$$
		x\mapsto f(z,x) + \hat{\xi}_\rho x^{p-1}
		$$
		is nondecreasing on $[0,\rho]$.
\end{itemize}

\begin{remark}
	Since we are looking for positive solutions and all the above hypotheses concern the positive semi-axis $\RR_+=[0,+\infty)$, we may assume without any loss of generality that $f(z,x)=0$ for almost all $z\in\Omega$, and for all $x\leq0$. Hypothesis $H(f)_1(ii)$ implies that $f(z,\cdot)$ is $(p-1)$-superlinear near $+\infty$. Hypothesis $H(f)_1(iii)$ implies that $f(z,\cdot)$ is $(q-1)$-superlinear near $0^+$ (that is, $f(z,\cdot)$ exhibits a $q$-concave term near $0^+$). Hypothesis $H(f)_1(iv)$ is satisfied if for example $f(z,\cdot)$ is differentiable and for every $\rho>0$, there exists $\eta_\rho>0$ such that $f_x'(z,x)x\geq-\eta_\rho x^{p-1}$ for almost all $z\in\Omega$, and for all $0\leq x\leq\rho$. We stress that no global sign condition is imposed on $f(z,\cdot)$.
\end{remark}

{\it Example}. The following function satisfies hypotheses $H(f)_1$. For the sake of simplicity we drop the $z$-dependence:
$$
f(z)=\left\{
\begin{array}{ll}
	0 & \mbox{if}\ x<0 \\
	x^{\tau-1}-2x^{q-1} & \mbox{if}\ 0\leq x\leq1 \\
	x^{r-1}-2x^{s-1} & \mbox{if}\ 1<x
\end{array}
\right.
$$
with $\tau<q\leq p$ and $1<s<r<p$. Note that $f(\cdot)$ changes sign.

Let $$\mathcal{L}=\{\lambda\in\RR:\ \mbox{problem \eqref{eqp} has a positive solution}\},$$
$$
S_\lambda = \mbox{the set of all positive solutions of problem}\ \eqref{eqp}.
$$

\begin{prop}\label{prop8}
	If hypotheses $H(a), H(\xi), H(\beta), H(f)_1$ hold, then $\mathcal{L}\neq\emptyset$ and $S_\lambda\subseteq D_+$.
\end{prop}
\begin{proof}
	Let $\eta>||\xi||_\infty$ and consider the following Carath\'eodory function
	\begin{equation}\label{eq5}
		e_\lambda(z,x) = \left\{
		\begin{array}{ll}
			0 & \mbox{if}\ x\leq0 \\
			\left(\lambda+\eta\right)x^{p-1} + f(z,x) & \mbox{if}\ 0<x,
		\end{array}
		\right.
	\end{equation}
for all $\lambda\in\RR$.

	We set $E_\lambda(z,x)=\int^x_0e_\lambda(z,s)ds$ and consider the $C^1$-functional $\varphi_\lambda:W^{1,p}(\Omega)\rightarrow\RR$ defined by
$$
\varphi_\lambda(u)=\frac{1}{p}\mu(u) + \frac{\eta}{p}||u||^p_p - \int_\Omega E_\lambda(z,u)dz\ \mbox{for all}\ u\in W^{1,p}(\Omega).
$$

Let $F(z,x)=\int^x_0f(z,s)ds$. Hypotheses $H(f)_1(i),(ii)$ imply that given $\epsilon>0$, we can find $c_6=c_6(\epsilon)>0$ such that
\begin{equation}\label{eq6}
	F(z,x)\leq\frac{\epsilon}{p}x^p+c_6\ \mbox{for almost all}\ z\in\Omega,\ \mbox{and for all}\ x\geq0.
\end{equation}

Using (\ref{eq5}), (\ref{eq6}), Corollary \ref{cor3} and hypothesis $H(\beta)$, we have
$$
\varphi_\lambda(u)\geq\frac{c_1}{p(p-1)}||Du||^p_p + \frac{1}{p}\int_\Omega\left[\xi(z)+\eta-(\lambda+\xi)\right]|u|^pdz - c_7\ \mbox{for some}\ c_7>0.
$$

Choosing $\lambda\in\RR$ such that $\lambda+\epsilon<\eta-||\xi||_\infty$, we can write
$$
\begin{array}{ll}
	& \varphi_\lambda(u) \geq\frac{c_1}{p(p-1)}||Du||^p_p + c_8||u||^p_p - c_7\ \mbox{for some}\ c_8>0, \\
	\Rightarrow & \varphi_\lambda(\cdot)\ \mbox{is coercive}.
\end{array}
$$

By the Sobolev embedding theorem and the compactness of the trace map we deduce that $\varphi(\cdot)$ is sequentially weak lower semicontinuous. So, by the Weierstrass-Tonelli theorem, we can find $u_\lambda\in W^{1,p}(\Omega)$ such that
\begin{equation}\label{eq7}
	\varphi_\lambda(u_\lambda) = \inf\left\{\varphi_\lambda(u):u\in W^{1,p}(\Omega)\right\}.
\end{equation}

Hypothesis $H(a)(iv)$ implies that given $\tilde{c}_0>\tilde{c}$, we can find $\delta\in(0,1)$ such that
\begin{equation}\label{eq8}
	G(y)\leq\frac{\tilde{c_0}}{q}|y|^q\ \mbox{for all}\ |y|\leq\delta.
\end{equation}

Hypothesis $H(f)(iii)$ implies that given any $\vartheta>0$, by choosing $\delta>0$ even smaller if necessary, we can also have
\begin{equation}\label{eq9}
	F(z,x)\geq\frac{\vartheta}{q}x^q\ \mbox{for almost all}\ z\in\Omega,\ \mbox{and for all}\ 0\leq x\leq\delta.
\end{equation}

Let $\hat\lambda_1=\hat\lambda_1(q,\xi_0,\beta_0)$ and $\hat{u}_1=\hat{u}_1(q,\xi_0,\beta_0)\in D_+$ with $\xi_0=\frac{1}{\tilde{c}_0}\xi$, $\beta_0=\frac{1}{\tilde{c}_0}\beta$. We choose small $t\in(0,1)$ such that
\begin{equation}\label{eq10}
	0\leq t\hat{u}_1(z)\leq\delta\ \mbox{and}\ |D(t\hat{u}_1)(z)|\leq\delta\ \mbox{for all}\ z\in\overline\Omega.
\end{equation}

Using (\ref{eq5}), (\ref{eq8}), (\ref{eq9}), (\ref{eq10}), we have
$$\begin{array}{ll}
\varphi_\lambda(t\hat{u}_1) &\di \leq\frac{\tilde{c}_0t^q}{q}||D\hat{u}_1||^q_q + \frac{1}{p}\int_\Omega \xi(z)|t\hat{u}_1|^pdz + \frac{1}{p}\int_{\partial\Omega}\beta(z)|t\hat{u}_1|^pd\sigma \\
&\di -\frac{\lambda t^p}{p}||\hat{u}_1||^p_p - \frac{\vartheta t^q}{q}\ \mbox{(recall that $||\hat{u}_1||_q=1$)} \\
&\di \leq\frac{\tilde{c}_0 t^q}{q}\left(||D\hat{u}_1||^q_q + \int_\Omega\xi_0|\hat{u}_1|^qdz + \int_{\partial\Omega}\beta_0(z)|\hat{u}_1|^qd\sigma\right) - \frac{\vartheta}{q}t^q \\
&\di \mbox{(since $0<\delta<1$ and $q\leq p$)} \\
&\di \leq\frac{t^q}{q}\left(\tilde{c}_0\hat\lambda_1-\vartheta\right).
\end{array}$$

But $\vartheta>0$ is arbitrary. So, choosing $\vartheta>\tilde{c}_0\hat\lambda_1$, we see that
$$
\begin{array}{ll}
	& \varphi_\lambda(t\hat{u}_1)<0, \\
	\Rightarrow & \varphi_\lambda(u_\lambda)<0 = \varphi_\lambda(0)\ \mbox{(see (\ref{eq7}))},\\
	\Rightarrow & u_\lambda\neq0.
\end{array}
$$

From (\ref{eq7}) we have for all $h\in W^{1,p}(\Omega)$
\begin{eqnarray}
	&& \varphi'_\lambda(u_\lambda)=0 \nonumber \\
	& \Rightarrow &\langle A(u_\lambda),h\rangle + \int_\Omega(\xi + \eta)|u_\lambda|^{p-2}u_\lambda hdz + \int_{\partial\Omega}\beta_\lambda|u_\lambda|^{p-2}u_\lambda hd\sigma = \int_\Omega e_\lambda(z,u_\lambda)hdz. \label{eq11}
\end{eqnarray}

In (\ref{eq11}) we choose $h=-u^-_{\lambda}\in W^{1,p}(\Omega)$. Then
$$
\begin{array}{ll}
	& \frac{c_1}{p-1}||Du^-_\lambda||^p_p + \frac{1}{p}\int_\Omega[\xi(z)+\eta](u^-_\lambda)^pdz\leq0 \\
	& \mbox{(see Lemma \ref{lem2}, hypothesis $H(\beta)$, and \eqref{eq5})} \\
	\Rightarrow & u_\lambda\geq0,\ u_\lambda\neq0\ \mbox{(recall that $\eta>||\xi||_\infty$)}.
\end{array}
$$

It follows  from (\ref{eq5}) and (\ref{eq11}) that
\begin{equation}\label{eq12}\begin{array}{ll}
&\displaystyle	-{\rm div}\,a(Du_\lambda(z)) + \xi(z)u_\lambda(z)^{p-1} = \lambda u_\lambda(z)^{p-1} + f(z,u_\lambda(z))\ \mbox{for almost all}\ z\in\Omega  \\
&\displaystyle	\frac{\partial u}{\partial n_a} + \beta(z)u^{p-1}_\lambda = 0\ \mbox{on}\ \partial\Omega\ \mbox{(see Papageorgiou \& R\u{a}dulescu \cite{17})}.
\end{array}
\end{equation}

From (\ref{eq12}) and Papageorgiou \& R\u{a}dulescu \cite{20}, we have $u_\lambda\in L^\infty(\Omega)$. Then the nonlinear regularity theory of Lieberman \cite{12} implies that $u_\lambda\in C_+\backslash\{0\}$.

Let $\rho=||u_\lambda||_\infty$ and let $\hat{\xi_\rho}>0$ be as postulated by hypothesis $H(f)(iv)$. Then from (\ref{eq35}) we have
$$
\begin{array}{ll}
	& {\rm div}\,a(Du_\lambda(z))\leq \left(||\xi||_\infty + \hat\xi_\rho\right)u_\lambda(z)^{p-1}\ \mbox{for almost all}\ z\in\Omega, \\
	\Rightarrow & u_\lambda\in D_+\ \mbox{(see Pucci \& Serrin \cite[pp. 111, 120]{24})}.
\end{array}
$$

Therefore we conclude that $\lambda\in\mathcal{L}$ and so $\mathcal{L}\neq\emptyset$ and also $S_\lambda\subseteq D_+$.
\end{proof}

Next, we show that $\mathcal{L}$ is a half-line.

\begin{prop}\label{prop9}
	If hypotheses $H(a),H(\xi),H(\beta),H(f)_1$ hold, $\lambda\in\mathcal{L}$, and $\vartheta<\lambda$, then $\vartheta\in\mathcal{L}$.
\end{prop}

\begin{proof}
	By hypothesis, $\lambda\in\mathcal{L}$. So, we can find $u_\lambda\in S_\lambda\subseteq D_+$. With $\eta>||\xi||_\infty$ as before, we introduce the following truncation-perturbation of the reaction in problem \eqref{eqp}:
	\begin{equation}\label{eq13}
		k_\vartheta(z,x)=\left\{
			\begin{array}{ll}
				0 & \mbox{if}\ x<0 \\
				\left(\vartheta+\eta\right)x^{p-1} + f(z,x) & \mbox{if}\ 0\leq x\leq u_\lambda(z) \\
				\left(\vartheta+\eta\right)u_\lambda(z)^{p-1} + f(z,u_\lambda(z)) & \mbox{if}\ u_\lambda(z)<x.
			\end{array}
		\right.
	\end{equation}
	
	This is a Carath\'eodory function. We set $K_\vartheta(z,x) = \int_0^x k_\vartheta(z,s)ds$ and consider the $C^1$-functional $\hat\varphi_\vartheta:W^{1,p}(\Omega)\rightarrow\RR$ defined by
	$$
	\hat\varphi_\vartheta(u)=\frac{1}{p}\mu(u) + \frac{\eta}{p}||u||^p_p - \int_\Omega K_\vartheta(z,u)dz\ \mbox{for all}\ u\in W^{1,p}(\Omega).
	$$
	
	Clearly, $\hat\varphi_\vartheta(\cdot)$ is coercive (see (\ref{eq13})) and sequentially weakly lower semicontinuous. So, we can find $u_\vartheta\in W^{1,p}(\Omega)$ such that
	\begin{equation}\label{eq14}
		\hat\varphi_\vartheta(u_\vartheta) = \inf\left\{\hat{\varphi}_\vartheta (u):u\in W^{1,p}(\Omega)\right\}.
	\end{equation}
	
	As in the proof of Proposition \ref{prop8}, using hypotheses $H(a)(iv)$ and $H(f)(iii)$, we show that $\hat\varphi_\vartheta(u_\vartheta)<0 = \hat\varphi_\vartheta(0)$, hence $u_\vartheta\neq0$. From (\ref{eq14}) we have
	\begin{eqnarray}
		& \hat\varphi'_\vartheta(u_\varphi) = 0, \nonumber \\
		\Rightarrow & \langle A(u_\vartheta),h\rangle + \int_\Omega\left(\xi(z)+\eta\right)|u_\vartheta|^{p-2}u_\vartheta hdz + \int_{\partial\Omega}\beta(z)|u_\vartheta|^{p-2}u_\vartheta hd\sigma \label{eq15} \\
		& = \int_\Omega k_\vartheta(z,u_\vartheta)hdz\ \mbox{for all}\ h\in W^{1,p}(\Omega) \nonumber.
	\end{eqnarray}
	
	In (\ref{eq15})  we first choose $h=-u^-_\vartheta\in W^{1,p}(\Omega)$. Then using Lemma 2 and (\ref{eq13}) we obtain
	$$
	\begin{array}{ll}
		& \frac{c_1}{p-1}||Du^-_\vartheta||^p_p + \int_\Omega\left[\xi(z)+\eta\right](u^-_\vartheta)^pdz \leq0\ \mbox{(see hypothesis $H(\beta)$)}, \\
		\Rightarrow & u_\vartheta\geq0,\ u_\vartheta\neq0.
	\end{array}
	$$
	
	Next, in (\ref{eq15}) we choose $h=(u_\vartheta-u_\lambda)^+\in W^{1,p}(\Omega)$. Then
	$$
	\begin{array}{ll}
		& \langle A(u_\vartheta),(u_\vartheta-u_\lambda)^+\rangle + \int_\Omega\left(\xi(z)+\eta\right)u^{p-1}_\vartheta(u_\vartheta-u_\lambda)^+dz + \int_{\partial\Omega}\beta(z)u^{p-1}_\vartheta(u_\vartheta-u_\lambda)^+d\sigma \\
		= & \int_\Omega(\left(\vartheta+\eta\right)u^{p-1}_\lambda + f(z,u_\lambda))(u_\vartheta-u_\lambda)^+dz\ \mbox{(see (\ref{eq13}))} \\
		\leq & \int_\Omega(\left(\lambda+\eta\right)u^{p-1}_\lambda + f(z,u_\lambda)(u_\vartheta-u_\lambda)^+dz\ \mbox{(recall that $\vartheta<\lambda$)} \\
		= & \langle A(u_\lambda), (u_\vartheta-u_\lambda)^+)\rangle + \int_\Omega\left(\xi(z)+\eta\right)u^{p-1}_\lambda(u_\vartheta-u_\lambda)^+dz + \int_{\partial\Omega}\beta(z)u^{p-1}_\lambda(u_\vartheta-u_\lambda)^+d\sigma\ \\
		& \mbox{(recall that $u_\lambda\in S_\lambda$)} \\
		\Rightarrow & u_\vartheta\leq u_\lambda.
	\end{array}
	$$
	
	We have proved that
	\begin{equation}\label{eq16}
		u_\vartheta\in[0,u_\lambda],\ u_\vartheta\neq0.
	\end{equation}
	
It follows	from (\ref{eq13}), (\ref{eq15}), (\ref{eq16}) that $\vartheta\in\mathcal{L}$ and $u_\vartheta\in S_\vartheta\subseteq D_+$.
\end{proof}

Let $\lambda^{*}=\sup\mathcal{L}$.

\begin{prop}\label{prop10}
	If hypotheses $h(a),H(\xi), H(\beta), H(f)_1$ hold, then $\lambda^{*}<+\infty$.
\end{prop}

\begin{proof}
	Hypotheses $H(\xi),H(f)_1$ imply that for large enough $\tilde\lambda>0$ we have
	\begin{equation}\label{eq17}
		(\tilde\lambda-\xi(z))x^{p-1} + f(z,x) \geq x^{p-1}\ \mbox{for almost all}\ z\in\Omega,\ \mbox{and for all}\ x\geq0.
	\end{equation}
	
	Let $\lambda>\tilde\lambda$ and suppose that $\lambda\in\mathcal{L}$. Then by Proposition \ref{prop8} we can find $u\in S_\lambda\subseteq D_+$. We set
	\begin{equation}\label{eq18}
		m=\min_{\overline\Omega}u>0\quad \mbox{(since $u\in D_+$)}.
	\end{equation}
	
	For $\delta>0$ we set $m_\delta=m+\delta>0$. Also, let $\rho=||u||_\infty$ and let $\hat\xi_\rho>0$ be as postulated by hypothesis $H(f)_1(iv)$. We can always take $\hat\xi_\rho>\max\{\lambda,||\xi||_\infty\}$. We have that for almost all $z\in\Omega$ the function $x\mapsto(\lambda+\hat\xi_\rho)x^{p-1} + f(z,x)$ is nondecreasing on $[0,\rho]$. We have
	$$
	\begin{array}{ll}
		& -{\rm div}\,a(Dm_\delta) + (\xi(z) + \hat\xi_\rho)m^{p-1}_\delta \\
		\leq & (\xi(z)+\hat\xi_\rho)m^{p-1} + \gamma(\delta)\ \mbox{with}\ \gamma(\delta)\rightarrow0^+\ \mbox{as}\ \delta\rightarrow0^+ \\
		\leq & (\tilde\lambda + \hat\xi_\rho)m^{p-1} + f(z,m) + \gamma(\delta)\ \mbox{(see (\ref{eq17}))} \\
		= & (\lambda+\hat\xi_\rho)m^{p-1} + f(z,m) - (\lambda-\tilde\lambda)m^{p-1} + \gamma(\delta) \\
		\leq & (\lambda + \hat\xi_\rho)m^{p-1} + f(z,m)\ \mbox{for small enough}\ \delta>0\ \mbox{(so that $\gamma(\delta)<(\lambda-\tilde\lambda)m^{p-1}$)} \\
		\leq & (\lambda+\hat\xi_\rho)u^{p-1} + f(z,u)\ \mbox{(see (\ref{eq18}))}\\
		= & -{\rm div}\,a(Du) + (\xi(z) + \hat\xi_\rho)u^{p-1}.
	\end{array}
	$$
	
	Let
	$$
	\begin{array}{ll}
		h_1(z)=(\xi(z)+\hat\xi_\rho)m^{p-1}+\gamma(\delta) \\
		h_2(z)=(\lambda+\hat\xi_\rho)u^{p-1} + f(z,u).
	\end{array}
	$$
	
	Evidently, $h_1,h_2\in L^\infty(\Omega)$ and for $\delta>0$ small we have
	$$
	0<\tilde\gamma\leq h_2(z)-h_1(z)\ \mbox{for almost all}\ z\in\Omega.
	$$
	
	So, by Proposition \ref{prop5} for small enough $\delta>0$ we have
	$$
	u-m_\delta\in {\rm int}\,C_+,
	$$
	a contradiction to (\ref{eq18}). Therefore $\lambda\not\in\mathcal{L}$ and so $\lambda^*\leq\tilde\lambda<+\infty$.
\end{proof}

Fix $\lambda<\lambda^*$. Then by Proposition \ref{prop9} we have $\lambda\in\mathcal{L}$. We will show that $S_\lambda\subseteq D_+$ admits a smallest element. Let $r\in(p,p^*)$. On account of hypotheses $H(f)_1$ we can find $c_9>0$ and $c_{10}=c_{10}(\lambda)>0$ both large enough such that
\begin{equation}\label{eq19}
	\lambda x^{p-1} + f(z,x)\geq c_9x^{q-1} - c_{10}x^{r-1}\ \mbox{for almost all}\ z\in\Omega,\ \mbox{and for all}\ x\geq0.
\end{equation}

Motivated by this one-sided growth condition on the reaction of problem \eqref{eqp}, we consider the following auxiliary nonlinear nonhomogeneous Robin problem
\begin{equation}\label{eq20}
	\left\{
	\begin{array}{ll}
		-{\rm div}\,a(Du(z)) + |\xi(z)|u^{p-1} = c_9u(z)^{q-1} - c_{10}u(z)^{r-1}\ \mbox{in}\ \Omega, \\
		\frac{\partial u}{\partial n_a} + \beta(z)u^{p-1}=0\ \mbox{on}\ \partial\Omega,\ u>0.
	\end{array}
	\right\}
\end{equation}

\begin{prop}\label{prop11}
	If hypotheses $H(a),H(\xi),H(\beta)$ hold and $c_9,c_{10}>0$ are both large enough, then problem (\ref{eq20}) admits a unique solution $u^\lambda_*\in D_+$
\end{prop}

\begin{proof}
	Let $\Psi:W^{1,p}(\Omega)\rightarrow\RR$ be the $C^1$-functional defined by
	$$
	\begin{array}{ll}
		\Psi(u)=\int_\Omega G(Du)dz + \frac{1}{p}\int_\Omega|\xi(z)||u|^pdz + \frac{1}{p}\int_{\partial\Omega}\beta(z)|u|^pd\sigma + \frac{1}{p}||u^-||^p_p \\
		+\frac{c_{10}}{r}||u^+||^r_r - \frac{c_9}{q}||u^+||^q_q\ \mbox{for all}\ u\in W^{1,p}(\Omega).
	\end{array}
	$$
	
	By Corollary \ref{cor3} and the fact that $q\leq p<r$, by taking $c_{10}>0$ large enough (see (\ref{eq19})), we see that $\Psi_\lambda(\cdot)$ is coercive. Also, it is sequentially weakly lower semicontinuous. So, we can find $u^\lambda_*\in W^{1,p}(\Omega)$ such that
	\begin{equation}\label{eq21}
		\Psi(u^\lambda_*)=\inf\left\{\Psi(u):u\in W^{1,p}(\Omega)\right\}.
	\end{equation}
	
	On account of hypothesis $H(f)(iii)$, we can choose $c_9>0$ large enough so that
	$$
	\begin{array}{ll}
		& \Psi(u^\lambda_*)<0 = \Psi(0)\ \mbox{(recall that $q\leq p<r$)} \\
		\Rightarrow & u^\lambda_*\neq0.
	\end{array}
	$$
	
	From (\ref{eq21}) we have
	\begin{equation}\label{eq22}
	\begin{array}{ll}
		& \Psi'(u^\lambda_*)=0, \\
		\Rightarrow & \langle A(u^\lambda_*),h\rangle + \int_\Omega|\xi(z)||u^\lambda_*|^{p-2}u^\lambda_*hdz + \int_{\partial\Omega}\beta(z)|u^\lambda_*|^{p-2}u^\lambda_*hd\sigma \\
		& -\int_\Omega((u^\lambda_*)^-)^{p-1}hdz \\
		= & c_9\int_\Omega((u^\lambda_*)^+)^{q-1}hdz - c_{10}\int_\Omega((u^\lambda_*)^+)^{r-1}hdz\ \mbox{for all}\ h\in W^{1,p}(\Omega).
	\end{array}
	\end{equation}
	
	In (\ref{eq22}) we choose $h=-(u^\lambda_*)^{-}\in W^{1,p}(\Omega)$. Then
	$$
	\begin{array}{ll}
		& \frac{c_1}{p-1}||D(u^\lambda_*)^-||^p_p + \int_\Omega(|\xi(z)|+1)((u^\lambda_*)^-)^{p}dz\leq0\ \mbox{(see Lemma \ref{lem2} and hypothesis $H(\beta)$)} \\
		\Rightarrow & u^\lambda_*\geq0,\ u^\lambda_*\neq0.
	\end{array}
	$$
	
	It follows  from (\ref{eq22}) that $u^\lambda_*$ is a positive solution of (\ref{eq20}). The nonlinear regularity theory implies that $u^\lambda_*\in C_+\backslash\{0\}$. Moreover, we have
	$$
	{\rm div}\,a(Du^\lambda_*(z))\leq\left(||\xi||_\infty + c_{10}||u^\lambda_*||^{r-p}_\infty\right)u^\lambda_*(z)\ \mbox{for almost all}\ z\in\Omega.
	$$
	
	Next, we show that this positive solution of (\ref{eq20}) is unique. For this purpose we introduce the functional $l:L^1(\Omega)\rightarrow\overline\RR=\RR\cup\{+\infty\}$ defined by
	$$
	l(u)=\left\{
	\begin{array}{ll}
	\int_\Omega G(Du^{\frac{1}{q}})dz + \frac{1}{p}\int_\Omega|\xi(z)|u^\frac{p}{q}dz + \frac{1}{p}\int_{\partial\Omega}\beta(z)u^\frac{p}{q}d\sigma & \mbox{if}\ u\geq0,\ u^\frac{1}{q}\in W^{1,p}(\Omega) \\
	+\infty & \mbox{otherwise}.
	\end{array}
	\right.
	$$
	
	Here, $q\leq p$ is as in hypothesis $H(a)(iv)$. Let ${\rm dom}\,l=\{u\in L^1(\Omega):l(u)<+\infty\}$ (the effective domain of $l(\cdot)$). Let $u_1,u_2\in {\rm dom}\,l$ and consider $u=[(1-t)u_1+tu_2]^\frac{1}{q}$ with $t\in[0,1]$. From Lemma 1 of Diaz \& Saa \cite{6}, we have
	$$
	\begin{array}{ll}
	& |Du(z)|\leq[(1-t)|Du_1(z)^\frac{1}{q}|^q + t|Du_2(z)^\frac{1}{q}|^q]^\frac{1}{q}\ \mbox{for almost all}\ z\in\Omega \\
	\Rightarrow & G_0(|Du(z)|)\leq G_0([(1-t)|Du_1(z)^\frac{1}{q}|^p + t|Du_2(z)^\frac{1}{q}|^q]^\frac{1}{q})\ \mbox{(since $G_0(\cdot)$ is increasing)} \\
	& \leq (1-t)G_0(|Du_1(z)^\frac{1}{q}|) + tG_0(|Du_2(z)^\frac{1}{q}|)\ \mbox{for almost all}\ z\in\Omega\ \\
	& \mbox{(see hypotheses $H(a)(iv)$)} \\
	\Rightarrow & G(Du(z))\leq(1-t)G(Du_1(z)^\frac{1}{q}) + tG(Du_2(z)^\frac{1}{q})\ \mbox{for almost all}\ z\in\Omega, \\
	\Rightarrow & {\rm dom}\,l\ni u\mapsto\int_\Omega G(Du^\frac{1}{q})dz\ \mbox{is convex}.
	\end{array}
	$$
	
	Since $q\leq p$ and $\beta\geq0$ (see hypotheses $H(\beta)$), we deduce that the mapping
	$$
	{\rm dom}\,l\ni u\mapsto\int_\Omega|\xi(z)|u^\frac{p}{q}dz + \int_{\partial\Omega}\beta(z)u^\frac{p}{q}d\sigma\ \mbox{is convex}.
	$$
	
	Therefore, we conclude that $l(\cdot)$ is convex and by Fatou's lemma it is also lower semicontinuous.
	
	Assume that $v^\lambda_*$ is another positive solution for problem (\ref{eq20}). Again, we can show that $v^\lambda_*\in D_+$. Let $h\in C^1(\overline\Omega)$. For $|t|<1$ small enough we have
	$$
	(u^\lambda_*)^q + th\in {\rm dom}\,l\ \mbox{and}\ (v^\lambda_*)^q + th\in {\rm dom}\,l.
	$$
	
	It is easily seen that $l(\cdot)$ is G\^ateaux differentiable at $(u^\lambda_*)^q$ and at $(v^\lambda_*)^q$ in the direction $h$. Using the chain rule and the nonlinear Green identity (see Gasinski \& Papageorgiou \cite[p. 210]{8}), we obtain
	$$
	\begin{array}{ll}
		l'((u^\lambda_*)^q)(h) = \frac{1}{q}\int_\Omega\frac{-{\rm div}\,a(Du^\lambda_*) + |\xi(z)|(u^\lambda_*)^{p-1}}{(u^\lambda_*)^{q-1}}hdz \\
		l'((v^\lambda_*)^q)(h) = \frac{1}{q}\int_\Omega\frac{-{\rm div}\,a(Dv^\lambda_*) + |\xi(z)|(v^\lambda_*)^{p-1}}{(v^\lambda_*)^{q-1}}hdz.
	\end{array}
	$$
	
	The convexity of $l(\cdot)$ implies the monotonicity of $l'(\cdot)$. Therefore
	$$
	\begin{array}{ll}
		0 & \leq \int_\Omega\left[\frac{-{\rm div}\,a(Du^\lambda_*) + |\xi(z)|(u^\lambda_*)^{p-1}|}{(u^\lambda_*)^{q-1}} - \frac{-{\rm div}\,a(Dv^\lambda_*) + |\xi(z)|(v^\lambda_*)^{p-1}|}{(v^\lambda_*)^{q-1}}\right]((u^\lambda_*)^q-(v^\lambda_*)^q)dz\\
		& = \int_\Omega c_{10}[(u^\lambda_*)^{r-q}-(v^\lambda_*)^{r-q}](u^\lambda_*)^q-(v^\lambda_*)^q)dz\ \mbox{(see (\ref{eq20}))} \\
		\Rightarrow & u^\lambda_*=v^\lambda_*\ \mbox{(recall that $q\leq p<r$).}
	\end{array}
	$$
	
	So, the positive solution $u^\lambda_*\in D_+$ of problem (\ref{eq20}) is unique.
\end{proof}

\begin{prop}\label{prop12}
	If hypotheses $H(a), H(\xi), H(\beta), H(f)_1$ hold and $\lambda<\lambda^*$, then $u^\lambda_*\leq u$ for all $u\in S_\lambda$.
\end{prop}

\begin{proof}
	From Proposition \ref{prop9} we know that $\lambda\in\mathcal{L}$. Let $u\in S_\lambda\subseteq D_+$ (see Proposition \ref{prop8}). Again we fix $\eta>||\xi||_\infty$ and consider the Carath\'eodory function $\vartheta:\Omega\times\RR\rightarrow\RR$ defined by
	\begin{equation}\label{eq23}
		\vartheta(z,x)=\left\{
		\begin{array}{ll}
			0 & \mbox{if}\ x<0 \\
			c_9x^{q-1} - c_{10}x^{r-1} + \eta x^{p-1} & \mbox{if}\ 0\leq x\leq u(z) \\
			k_9 u(z)^{q-1} - c_{10}u(z)^{r-1} +\eta u(z)^{p-1} & \mbox{if}\ u(z)<x.
		\end{array}
		\right.
	\end{equation}
	
	We set $\Theta(z,x)=\int^x_0\vartheta(z,s)ds$ and consider the $C^1$-functional $\zeta:W^{1,p}(\Omega)\rightarrow\RR$ defined by
	$$
	\zeta(u) = \frac{1}{p}\mu(u) + \frac{\eta}{p}||u||^p_p - \int_\Omega\Theta(z,u)dz\ \mbox{for all}\ u\in W^{1,p}(\Omega).
	$$
	
	As before, $\zeta(\cdot)$ is coercive and sequentially weakly lower semicontinuous. So, we can find $\tilde{u}^\lambda_*\in W^{1,p}(\Omega)$ such that
	\begin{equation}\label{eq24}
		\zeta(\tilde{u}^\lambda_*)=\inf\{\zeta(u):u\in W^{1,p}(\Omega)\}.
	\end{equation}
	
	Since $q\leq p<r$, for $c_9, c_{10}>0$ large enough as in Proposition \ref{prop10}, we have
	$$
	\begin{array}{ll}
		& \zeta(\tilde{u}^\lambda_*)<0=\zeta(0), \\
		\Rightarrow & \tilde{u}^\lambda_*\neq0.
	\end{array}
	$$
	
	From (\ref{eq24}) we have
	\begin{equation}\label{eq25}
		\begin{array}{ll}
		& \zeta'(\tilde{u}^\lambda_*)=0, \\
		\Rightarrow & \langle A(\tilde{u}^\lambda_*),h\rangle + \int_\Omega[\xi(z)+\eta]|\tilde{u}^\lambda_*|^{p-2}\tilde{u}^\lambda_*hdz + \int_{\partial\Omega}\beta(z)|\tilde{u}^\lambda_*|^{p-2}\tilde{u}^\lambda_*hd\sigma \\
		& = \int_\Omega\vartheta(z,\tilde{u}^\lambda_*)hdz\ \mbox{for all}\ h\in W^{1,p}(\Omega).
		\end{array}		
	\end{equation}
	
	Let $h=-(\tilde{u}^\lambda_*)^-\in W^{1,p}(\Omega)$ in (\ref{eq25}). Then
	$$
	\begin{array}{ll}
		& \frac{c_1}{p-1}||D(\tilde{u}^\lambda_*)^-||^p_p + \int_\Omega[\xi(z)+\eta]((\tilde{u}^\lambda_*)^-)^pdz\leq0\ \mbox{(see Lemma \ref{lem2}, hypothesis $H(\beta)$, and (\ref{eq23}))} \\
		\Rightarrow & \tilde{u}^\lambda_*\geq0,\ \tilde{u}^\lambda_*\neq0.
	\end{array}
	$$
	
	Also, let $h=(\tilde{u}^\lambda_*-u)^+\in W^{1,p}(\Omega)$ in (\ref{eq25}). Then
	\begin{eqnarray*}
		&& \langle A(\tilde{u}^\lambda_*),(\tilde{u}^\lambda_*-u)^+\rangle + \int_\Omega[\xi(z)+\eta](\tilde{u}^\lambda_*)^{p-1}(\tilde{u}^\lambda_*-u)^+dz + \int_{\partial\Omega}\beta(z)(\tilde{u}^\lambda_*)^{p-1}(\tilde{u}^\lambda_*-u)^+d\sigma \\
		&= & \int_\Omega(c_9u^{q-1}-c_{10}u^{r-1}+\eta u^{p-1})(\tilde{u}^\lambda_*-u)^+dz\ \mbox{(see (\ref{eq23}))} \\
		&\leq & \int_\Omega [(\lambda+\eta)u^{p-1} + f(z,u)](\tilde{u}^\lambda_*-u)^+dz\ \mbox{(see (\ref{eq19}))} \\
		&= & \langle A(u),(\tilde{u}^\lambda_*-u)^+\rangle + \int_\Omega[\xi(z)+\eta]u^{p-1}(\tilde{u}^\lambda_*-u)dz + \int_{\partial\Omega}\beta(z)u^{p-1}(\tilde{u}^\lambda_*-u)^+d\sigma\\
		&&\mbox{(since $u\in S_\lambda$)}\\
		\Rightarrow && \tilde{u}^\lambda_*\leq u.
	\end{eqnarray*}
	
	So, we have proved that
	\begin{equation}\label{eq26}
		\tilde{u}^\lambda_*\in[0,u],\ \tilde{u}^\lambda_*\neq0.
	\end{equation}
	
	Then from (\ref{eq23}), (\ref{eq25}), (\ref{eq26}) we infer that
	$$
	\begin{array}{ll}
		& \tilde{u}^\lambda_*\ \mbox{is a positive solution of (\ref{eq20})}, \\
		\Rightarrow & \tilde{u}^\lambda_*=u^\lambda_*\in D_+\ \mbox{(see Proposition \ref{prop11})}, \\
		\Rightarrow & u^\lambda_*\leq u\ \mbox{for all}\ u\in S_\lambda\ \mbox{(see (\ref{eq26}))}.
	\end{array}
	$$
The proof is complete.
\end{proof}

From Papageorgiou, R\u{a}dulescu \& Repov\v{s} \cite{21} (proof of Proposition \ref{prop7}), we know that the set $S_\lambda$ is downward directed (that is, if $u_1,u_2\in S_\lambda$, then we can find $u\in S_\lambda$ such that $u\leq u_1,u\leq u_2$).

\begin{prop}\label{prop13}
	If hypotheses $H(a), H(\xi), H(\beta), H(f)_1$ hold and $\lambda<\lambda^*$, then $S_\lambda$ admits a smallest element $\overline{u}_\lambda\in D_+$, that is,
	$$
	\overline{u}_\lambda\leq u\ \mbox{for all}\ u\in S_\lambda.
	$$
\end{prop}

\begin{proof}
	According to Lemma 3.10 of Hu \& Papageorgiou \cite[p. 178]{11}, we can find $\{u_n\}_{n\geq1}\subseteq S_\lambda$ such that
	$$
	\inf S_\lambda=\inf_{n\geq1}u_n.
	$$
	
	Moreover, since $S_\lambda$ is downward directed, we can choose $\{u_n\}_{n\geq1}\subseteq S_\lambda$ to be decreasing. We have
	\begin{equation}\label{eq27}
	\begin{array}{ll}
		\langle A(u_n),h\rangle + \int_\Omega\xi(z)u_n^{p-1}hdz + \int_{\partial\Omega}\beta(z)u^{p-1}_nhd\sigma = \lambda\int_\Omega u_n^{p-1}hdz +  \\
		\int_\Omega f(z,u_n)hdz\ \mbox{for all}\ h\in W^{1,p}(\Omega),\ \mbox{all}\ n\in\NN
	\end{array}
	\end{equation}
	\begin{equation}\label{eq28}
	0\leq u_n\leq u_1\ \mbox{for all}\ n\in\NN.
	\end{equation}
	
	From (\ref{eq27}), (\ref{eq28}), we infer that $\{u_n\}_{n\geq1}\subseteq W^{1,p}(\Omega)$ is bounded. So, we may assume that
	\begin{equation}\label{eq29}
		u_n\xrightarrow{w}\overline{u}_\lambda\ \mbox{in}\ W^{1,p}(\Omega)\ \mbox{and}\ u_n\rightarrow\overline{u}_\lambda\ \mbox{in}\ L^p(\Omega)\ \mbox{and in}\ L^p(\partial\Omega).
	\end{equation}
	
	In (\ref{eq27}) we choose $h=u_n-\overline{u}_\lambda\in W^{1,p}(\Omega)$, pass to the limit as $n\rightarrow\infty$ and use (\ref{eq29}). Then
	\begin{equation}\label{eq30}
		\begin{array}{ll}
			& \lim_{n\rightarrow\infty} \langle A(u_n), u_n-\overline{u}_\lambda\rangle=0, \\
			\Rightarrow & u_n\rightarrow\overline{u}_\lambda\ \mbox{in}\ W^{1,p}(\Omega)\ \mbox{(see Proposition \ref{prop4})}.
		\end{array}
	\end{equation}
	
	In (\ref{eq27}) we pass to the limit as $n\rightarrow\infty$ and use (\ref{eq30}). Then
	\begin{equation}\label{eq31}
	\begin{array}{ll}
		\langle A(\overline{u}_\lambda),h\rangle + \int_\Omega\xi(z)\overline{u}^{p-1}_\lambda hdz + \int_{\partial\Omega}\beta(z)\overline{u}^{p-1}_\lambda hd\sigma = \lambda\int_\Omega\overline{u}^{p-1}_\lambda hdz \\
		+ \int_\Omega f(z,\overline{u}_\lambda) hdz\ \mbox{for all}\ h\in\ W^{1,p}(\Omega).		
	\end{array}
	\end{equation}
	
	Moreover, from Proposition \ref{prop11} we have
	$$
	\begin{array}{ll}
		& u^\lambda_*\leq u_n\ \mbox{for all}\ n\in\NN, \\
		\Rightarrow & u^\lambda_*\leq \overline{u}_\lambda, \mbox{hence}\ \overline{u}_\lambda\neq0, \\
		\Rightarrow & \overline{u}_\lambda\in S_\lambda\ \mbox{(see (\ref{eq31})) and}\ \overline{u}_\lambda=\inf S_\lambda.
	\end{array}
	$$
The proof is now complete.
\end{proof}

In the next proposition we establish the monotonicity and continuity properties of the map $\mathcal{L}\ni\lambda\mapsto u_\lambda\in C^1(\overline\Omega)$.

\begin{prop}\label{prop14}
	If hypotheses $H(a), H(\xi), H(\beta), H(f)_1$ hold, then the map $\lambda\mapsto\overline{u}_\lambda$ from $\mathcal{L}$ into $C^1(\overline\Omega)$ is:
	\begin{itemize}
		\item [(a)] strictly increasing in the sense that
			$$
			\vartheta<\lambda\Rightarrow\overline{u}_\lambda-\overline{u}_\vartheta\in {\rm int}\, C_+;
			$$
		\item[(b)] left continuous, that is, if $\lambda_n\rightarrow\lambda^-$ with $\lambda\in\mathcal{L}$, then $\bar{u}_{\lambda_n}\rightarrow\bar{u}_{\lambda}$ in $C^1(\overline{\Omega})$.
	\end{itemize}
\end{prop}
\begin{proof}
	$(a)$ Let $\vartheta<\lambda\in\mathcal{L}$. Let $\bar{u}_{\lambda_n}\in S_{\lambda}\subseteq D_+$ be the minimal solution of \eqref{eqp} (see Proposition \ref{prop12}). From Proposition \ref{prop9} and its proof we know that $\vartheta\in\mathcal{L}$ and we can find $u_{\vartheta}\in S_{\vartheta}\subseteq D_+$ such that $u_{\vartheta}\leq\bar{u}_{\lambda}$ (see (\ref{eq16})). Therefore $\bar{u}_{\vartheta}\leq\bar{u}_{\lambda}$.

Let $\rho=||\bar{u}_{\lambda}||_{\infty}$ and let $\hat{\xi}_{\rho}>0$ be as postulated by hypothesis $H(f)_1(iv)$. We can always take $\hat{\xi}_{\rho}>||\xi||_{\infty}$. Then
	\begin{eqnarray*}
		&&-{\rm div}\,a(D\bar{u}_{\vartheta})+[\xi(z)+\hat{\xi}_{\rho}]\bar{u}_{\vartheta}^{p-1}\\
		&=&\vartheta\bar{u}_{\vartheta}^{p-1}+f(z,\bar{u}_{\vartheta})+\hat{\xi}_{\rho}\bar{u}_{\vartheta}^{p-1}\\
		&=&\lambda\bar{u}_{\vartheta}^{p-1}+f(z,\bar{u}_{\vartheta})+\hat{\xi}_{\rho}\bar{u}_{\vartheta}^{p-1}-(\lambda-\vartheta)\bar{u}_{\vartheta}^{p-1}\\
		&\leq&\lambda\bar{u}_{\vartheta}^{p-1}+f(z,\bar{u}_{\vartheta})+\hat{\xi}_{\rho}\bar{u}_{\vartheta}^{p-1}-(\lambda-\vartheta)m^{p-1}_{\vartheta}\ \mbox{with}\ m_{\vartheta}=\min\limits_{\overline{\Omega}}\bar{u}_{\vartheta}>0\\
		&&(\mbox{recall that}\ \bar{u}_{\vartheta}\in D_+)\\
		&<&\lambda\bar{u}_{\lambda}^{p-1}+f(z,\bar{u}_{\lambda})+\hat{\xi}_{\rho}\bar{u}_{\lambda}^{p-1}\ (\mbox{since}\ \bar{u}_{\vartheta}\leq\bar{u}_{\lambda})\\
		&=&-{\rm div}\,a(D\bar{u}_{\lambda})+[\xi(z)+\hat{\xi}_{\rho}]\bar{u}_{\lambda}^{p-1}\ \mbox{for almost all}\ z\in\Omega.
	\end{eqnarray*}
	
	Let $$h_1(z)=\vartheta\bar{u}_{\lambda}^{p-1}+f(z,\bar{u}_{\vartheta})+\hat{\xi}_{\rho}\bar{u}_{\vartheta}^{p-1}$$
	and $$h_2(z)=\lambda\bar{u}_{\lambda}^{p-1}+f(z,\bar{\lambda})+\hat{\xi}_{\rho}\bar{u}_{\lambda}^{p-1}.$$
	
	Evidently, $h_1,h_2\in L^{\infty}(\Omega)$ (see hypothesis $H(f)_1(i)$) and
	$$0<(\lambda-\vartheta)m^{p-1}_{\vartheta}\leq h_2(z)-h_1(z)\ \mbox{for almost all}\ z\in\Omega.$$
	
	So, we can apply Proposition \ref{prop5} and conclude that $\bar{u}_{\lambda}-\bar{u}_{\vartheta}\in {\rm int}\, C_+$. This proves that the mapping $\lambda\mapsto\bar{u}_{\lambda}$ is strictly increasing.
	
	$(b)$ Let $\lambda_n\rightarrow \lambda^-$ with $\lambda\in\mathcal{L}$. We set $\bar{u}_n=\bar{u}_{\lambda_n}\in S_{\lambda_n}\subseteq D_+$ for all $n\in\NN$. Evidently, $\{\bar{u}_n\}_{n\geq 1}\subseteq W^{1,p}(\Omega)$ is bounded. So, we may assume that
	\begin{equation}\label{eq32}
		\bar{u}_n\stackrel{w}{\rightarrow}u_{\lambda}\ \mbox{in}\ W^{1,p}(\Omega)\ \mbox{and}\ \bar{u}_n\rightarrow u_{\lambda}\ \mbox{in}\ L^p(\Omega)\ \mbox{and in}\ L^p(\partial\Omega).
	\end{equation}
	
	We have
	\begin{eqnarray}\label{eq33}
		&&\left\langle A(\bar{u}_n),h\right\rangle+\int_{\Omega}\xi(z)\bar{u}_n^{p-1}hdz+\int_{\partial\Omega}\beta(z)\bar{u}_n^{p-1}hd\sigma=\int_{\Omega}
[\lambda_n\bar{u}_n^{p-1}+f(z,\bar{u}_n)]hdz\\
		&&\mbox{for all}\ h\in W^{1,p}(\Omega),\  n\in\NN\nonumber.
	\end{eqnarray}
	
	In (\ref{eq33}) we choose $h=\bar{u}_n-u_{\lambda}\in W^{1,p}(\Omega)$, pass to the limit as $n\rightarrow\infty$ and use (\ref{eq32}). Then
	\begin{eqnarray}\label{eq34}
		&&\lim\limits_{n\rightarrow\infty}\left\langle A(\bar{u}_n),\bar{u}_n-u_{\lambda}\right\rangle=0,\nonumber\\
		&\Rightarrow&\bar{u}_n\rightarrow u_{\lambda}\ \mbox{in}\ W^{1,p}(\Omega)\ (\mbox{see Proposition \ref{prop4}}).
	\end{eqnarray}
	
	So, if in (\ref{eq32}) we pass to the limit as $n\rightarrow\infty$ and use (\ref{eq34}), then we can infer that $u_{\lambda}\in S_{\lambda}\subseteq D_+$. On account of (\ref{eq34}) and Proposition 7 of Papageorgiou \& R\u{a}dulescu \cite{20}, we can find $c_{11}>0$ such that
	$$||\bar{u}_n||_{\infty}\leq c_{11}\ \mbox{for all}\ n\in\NN.$$
	
	Then the nonlinear regularity theory of Lieberman \cite{12} implies that there exist $\tau\in(0,1)$ and $c_{12}>0$ such that
	$$\bar{u}_n\in C^{1,\tau}(\overline{\Omega})\ \mbox{and}\ ||\bar{u}_n||_{C^{1,\tau}(\overline{\Omega})}\leq c_{12}\ \mbox{for all}\ n\in\NN.$$
	
	The existence of a compact embedding of $C^{1,\tau}(\overline{\Omega})$ into $C^1(\overline{\Omega})$ and (\ref{eq34}), imply that
	\begin{equation}\label{eq35}
		\bar{u}_n\rightarrow u_{\lambda}\ \mbox{in}\ C^1(\overline{\Omega})\ \mbox{as}\ n\rightarrow\infty.
	\end{equation}
	
	We show that $u_{\lambda}=\bar{u}_{\lambda}$. Arguing by contradiction, suppose that $u_{\lambda}\neq\bar{u}_{\lambda}$. Then we can find $z_0\in\overline{\Omega}$ such that
	\begin{eqnarray*}
		&&\bar{u}_{\lambda}(z_0)<u_{\lambda}(z_0),\\
		&\Rightarrow&\bar{u}_{\lambda}(z_0)<\bar{u}_n(z_0)\ \mbox{for all}\ n\geq n_0\ (\mbox{see (\ref{eq35})}),
	\end{eqnarray*}
	which contradicts (a). Therefore the mapping $\lambda\mapsto\bar{u}_{\lambda}$ is left continuous.
\end{proof}

Now we ready to show the non-admissibility of $\lambda^*$.
\begin{prop}\label{prop15}
	If hypotheses $H(a),H(\xi),H(\beta),H(f)_1$ hold, then $\lambda^*\notin\mathcal{L}$.
\end{prop}
\begin{proof}
	Arguing by contradiction, suppose that $\lambda^*\in\mathcal{L}$. According to Proposition \ref{prop13}, problem \eqref{eqp} admits a smallest positive solution $\bar{u}_*=\bar{u}_{\lambda^*}\in D_+$. Let $\lambda>\lambda^*,\eta>||\xi||_{\infty}$ and consider the Carath\'eodory function
	\begin{eqnarray}\label{eq36}
		\gamma_{\lambda}(z,x)=\left\{\begin{array}{lll}
			&(\lambda+\eta)\bar{u}_*(z)^{p-1}+f(z,\bar{u}_*(z))& \mbox{if}\ x\leq\bar{u}_*(z)\\
			&(\lambda+\eta)x^{p-1}+f(z,x)&\mbox{if} \ \bar{u}_*(z)<x.
		\end{array}\right.
	\end{eqnarray}
	
	We set $\Gamma_{\lambda}(z,x)=\int^x_0\gamma_{\lambda}(z,s)ds$ and consider the $C^1$-functional $\tilde{\varphi}_{\lambda}:W^{1,p}(\Omega)\rightarrow\RR$ defined by
	$$\tilde{\varphi}_{\lambda}(u)=\frac{1}{p}\mu(u)+\frac{\eta}{p}||u||^p_p-\int_{\Omega}\Gamma_{\lambda}(z,u)dz\ \mbox{for all}\ u\in W^{1,p}(\Omega).$$
	
	As in the proof of Proposition \ref{prop8}, using hypothesis $H(f)_1(ii)$, we show that $\tilde{\varphi}_{\lambda}(\cdot)$ is coercive. Moreover, $\tilde{\varphi}_{\lambda}(\cdot)$ is sequentially weakly lower semicontinuous. So, we can find $u_{\lambda}\in W^{1,p}(\Omega)$ such that
	\begin{eqnarray}\label{eq37}
		&&\tilde{\varphi}_{\lambda}(u_{\lambda})=\inf\{\tilde{\varphi}_{\lambda}(u):u\in W^{1,p}(\Omega)\},\nonumber\\
		&\Rightarrow&\tilde{\varphi}'_{\lambda}(u_{\lambda})=0,\nonumber\\
		&\Rightarrow&\left\langle A(u_{\lambda}),h\right\rangle+\int_{\Omega}[\xi(z)+\eta]|u_{\lambda}|^{p-2}u_{\lambda}hdz+\int_{\partial\Omega}\beta(z)|u_{\lambda}|^{p-2}u_{\lambda}hd\sigma=\int_{\Omega}\gamma_{\lambda}(z,u_{\lambda})hdz\\
		&&\mbox{for all}\ h\in W^{1,p}(\Omega).\nonumber
	\end{eqnarray}
	
	In (\ref{eq37}) we choose $h=(\bar{u}_*-u_{\lambda})^+\in W^{1,p}(\Omega)$. Then
	\begin{eqnarray*}
		&&\left\langle A(u_{\lambda}),(\bar{u}_*-u_{\lambda})^+\right\rangle+\int_{\Omega}[\xi(z)+\eta]|u_{\lambda}|^{p-2}u_{\lambda}(\bar{u}_*-u_{\lambda})^+dz\\
		&&\hspace{1cm}+\int_{\partial\Omega}\beta(z)|u_{\lambda}|^{p-2}u_{\lambda}(\bar{u}_*-u_{\lambda})^+d\sigma\\
		&=&\int_{\Omega}[(\lambda+\eta)\bar{u}_*^{p-1}+f(z,\bar{u}_*)](\bar{u}_*-u_{\lambda})^+dz\ (\mbox{see (\ref{eq36})})\\
		&\geq&\int_{\Omega}[(\lambda^*+\eta)\bar{u}_*^{p-1}+f(z,\bar{u}_*)](\bar{u}_*-u_{\lambda})^+dz\ (\mbox{since}\ \lambda>\lambda^*)\\
		&=&\left\langle A(\bar{u}_*),(\bar{u}_*-u_{\lambda})^+\right\rangle+\int_{\Omega}[\xi(z)+\eta]\bar{u}_*^{p-1}(\bar{u}_*-u_{\lambda})^+dz+\int_{\partial\Omega}\beta(z)\bar{u}_*^{p-1}(\bar{u}_*-u_{\lambda})^+d\sigma\\
		&&(\mbox{since}\ \bar{u}_*\in S_{\lambda^*}),\\
		&\Rightarrow&\bar{u}_*\leq u_{\lambda}.
	\end{eqnarray*}
	
	Then from (\ref{eq36}) and (\ref{eq37}) it follows that $\bar{u}_{\lambda}\in S_{\lambda}$ and so $\lambda\in\mathcal{L}$, a contradiction. This proves that $\lambda^*\notin\mathcal{L}$.
\end{proof}

So, summarizing the situation for problem \eqref{eqp} when the perturbation $f(z,\cdot)$ is $(p-1)$-sublinear, we can state the following theorem.
\begin{theorem}\label{th16}
	If hypotheses $H(a), H(\xi),H(\beta),H(f)_1$ hold, then there exists $\lambda^*<+\infty$ such that
	\begin{itemize}
		\item[(a)] for every $\lambda\geq\lambda^*$, problem \eqref{eqp} has no positive solutions;
		\item[(b)] for every $\lambda<\lambda^*$, problem \eqref{eqp} has at least one positive solution $u_{\lambda}\in D_+$;
		\item[(c)] for every $\lambda<\lambda^*$, problem \eqref{eqp} has a smallest positive solution $\bar{u}_{\lambda}\in D_+$ and the map $\lambda\mapsto\bar{u}_{\lambda}$ from $(-\infty,\lambda^*)$ into $C^1(\overline{\Omega})$ is
		
			 $\bullet$ strictly increasing, that is, if $\vartheta<\lambda<\lambda^*$, then
			$$\bar{u}_{\lambda}-\bar{u}_{\vartheta}\in {\rm int}\, C_+;$$
			
			 $\bullet$ left continuous, that is, if $\lambda_n\rightarrow\lambda^-$ and $\lambda<\lambda^*$, then $\bar{u}_{\lambda_n}\rightarrow\bar{u}_{\lambda}$ in $C^1(\overline{\Omega})$.
	\end{itemize}
\end{theorem}

In the special case of the $p$-Laplacian (that is, $a(y)=|y|^{p-2}y$ for all $y\in\RR^N$ with $1<p<\infty$), we can identify $\lambda^*$ as $\hat{\lambda}_1(p,\xi,\beta)$, when $f(z,x)>0$ for almost all $z\in\Omega$, and for all $x>0$.

So, we consider the following nonlinear Robin problem:
\begin{equation}\tag{$PL_{\lambda}$}\label{eql}
	\left\{\begin{array}{l}
		-\Delta_pu(z)+\xi(z)u(z)^{p-1}=\lambda u(z)^{p-1}+f(z,u(z))\ \mbox{in}\ \Omega,\\
		\frac{\partial u}{\partial n_p}+\beta(z)u^{p-1}=0\ \mbox{on}\ \partial\Omega,u>0,\lambda\in\RR,1<p<\infty.
	\end{array}\right\}
\end{equation}
\begin{prop}\label{prop17}
	Assume that hypotheses $H(\xi),H(\beta)$ hold and let $f:\Omega\times\RR\rightarrow\RR$ be a Carath\'eodory function such that
	\begin{itemize}
		\item for almost all $z\in\Omega$, $f(z,0)=0$ and $f(z,x)>0$ for all $x>0$;
		\item $f(z,x)\leq a(z)(1+x^{p^*-1})$ for almost all $z\in\Omega$, and for all $x\geq 0$, with $a\in L^{\infty}(\Omega)$.
	\end{itemize}
	Then for all $\lambda\geq\hat{\lambda}_1(p,\xi,\beta)$, $S_{\lambda}=\emptyset$.
\end{prop}
\begin{proof}
	Arguing by contradiction, suppose that $S_{\lambda}\neq\emptyset$ and let $u\in S_{\lambda}$. The nonlinear regularity theory implies that $u\in D_+$. Let $\hat{u}_1=\hat{u}_1(p,\xi,\beta)\in D_+$ (see Proposition \ref{prop7}). We consider the function
	$$R(\hat{u}_1,u)(z)=|D\hat{u}_1(z)|^p-|Du(z)|^{p-2}(Du(z),D\left(\frac{\hat{u}_1^p}{u^{p-1}}\right)(z))_{\RR^N}.$$
	
	The nonlinear Picone identity of Allegretto \& Huang \cite{2} implies that
	\begin{eqnarray*}
		&&0\leq R(\hat{u}_1,u)(z)\ \mbox{for almost all}\ z\in\Omega,\\
		&\Rightarrow&0\leq\int_{\Omega}R(\hat{u}_1,u)dz\\
		&&=||D\hat{u}_1||^p_p-\int_{\Omega}|Du|^{p-2}\left(Du,D\left(\frac{\hat{u}_1^p}{u^{p-1}}\right)\right)_{\RR^N}dz\\
		&&=||D\hat{u}_1||^p_p-\int_{\Omega}(-\Delta_pu)\frac{\hat{u}_1^p}{u^{p-1}}dz+\int_{\partial\Omega}\beta(z)u^{p-1}\frac{\hat{u}_1^p}{u^{p-1}}d\sigma
	\end{eqnarray*}
	(using the nonlinear Green identity, see Gasinski \& Papageorgiou \cite[p. 211]{8})
	\begin{eqnarray*}
		&&=||D\hat{u}_1||^p_p+\int_{\Omega}\xi(z)\hat{u}^{p-1}_1dz+\int_{\partial\Omega}\beta(z)\hat{u}_1^pd\sigma-\lambda-\int_{\Omega}f(z,u)\frac{\hat{u}_1^p}{u^{p-1}}\\
		&&(\mbox{see \eqref{eql} and recall that}\ ||\hat{u}_1||_p=1)\\
		&&<\mu(\hat{u}_1)-\lambda\ (\mbox{recall that}\ f(z,x)>0\ \mbox{for almost all}\ z\in\Omega,\ \mbox{and for all}\ x>0)\\
		&&=\hat{\lambda}_1-\lambda<0,
	\end{eqnarray*}
	a contradiction. Therefore $S_{\lambda}=\emptyset$ and so $\lambda\notin\mathcal{L}$ for all $\lambda\geq\hat{\lambda}_1(p,\xi,\beta)$.
\end{proof}

Moreover, reasoning as in the proof of Proposition \ref{prop8}, via the direct method of the calculus of variations, we obtain the following result. Note that now in hypothesis $H(a)(iv)$ we take  $q=p$.
\begin{prop}\label{prop18}
	If hypotheses $H(\xi),H(\beta),H(f)_1$ hold and $\lambda<\hat{\lambda}_1=\hat{\lambda}_1(p,\xi,\beta)$, then $\lambda\in\mathcal{L}$.
\end{prop}

We introduce the following stronger version of hypotheses $H(f)_1$.

\smallskip
$H(f)'_1:$ $f:\Omega\times\RR\rightarrow\RR$ is a Carath\'eodory function such that for almost all $z\in\Omega$, $f(z,0)=0$, $f(z,x)>0$ for all $x>0$ and hypotheses $H(f)'_1(i),\, (ii),\, (iii),\, (iv)$ are the same as the corresponding hypotheses $H(f)_1(i),\, (ii),\, (iii),\,(iv)$.

\smallskip
Using this stronger version of $H(f)_1$ and combining Propositions \ref{prop17} and \ref{prop18} we have the following theorem concerning the positive solutions of \eqref{eql} as the parameter $\lambda\in\RR$ varies.
\begin{theorem}\label{th19}
	If hypotheses $H(\xi),H(\beta),H(f)'_1$ hold, then
	\begin{itemize}
		\item[(a)] for every $\lambda\geq\hat{\lambda}_1=\hat{\lambda}_1(p,\xi,\beta)$, problem \eqref{eql} has no positive solutions;
		\item[(b)] for every $\lambda<\hat{\lambda}$, problem \eqref{eql} has at least one positive solution $u_{\lambda}\in D_+$;
		\item[(c)] for every $\lambda<\hat{\lambda}$, problem \eqref{eql} has a smallest positive solution $\bar{u}_{\lambda}\in D_+$ and the map $\lambda\mapsto\bar{u}_{\lambda}$ from $(-\infty,\hat{\lambda}_1)$ into $C^1(\overline{\Omega})$ is
		\begin{itemize}
			\item[$\bullet$] strictly increasing (that is, $\vartheta<\lambda<\hat{\lambda}_1\Rightarrow\bar{u}_{\lambda}-\bar{u}_{\vartheta}\in {\rm int}\, C_+$);
			\item[$\bullet$] left continuous.
		\end{itemize}
	\end{itemize}
\end{theorem}

If we  further restrict the conditions on the perturbation $f(z,x)$, we can have uniqueness for the positive solution.

The new hypotheses on $f(z,x)$ are the following:

\smallskip
$H(f)_1{''}:$ $f:\Omega\times\RR\rightarrow\RR$ is a Carath\'eodory function such that for almost all $z\in\Omega$, $f(z,0)=0, f(z,x)>0$ for all $x>0$, hypotheses $H(f)_1{''}(i),\,(ii),\,(iii),\,(iv)$ are the same as the corresponding hypotheses $H(f)_1(i),\,(ii),\,(iii),\,(iv)$ and
\begin{itemize}
	\item[(v)] if $x-y\geq m>0$, then $\frac{f(z,y)}{y^{p-1}}-\frac{f(z,x)}{x^{p-1}}\geq c_m>0$ for almost all $z\in\Omega$.
\end{itemize}
\begin{prop}\label{prop20}
	If hypotheses $H(\xi),H(\beta),H(f)_1''$ hold and $\lambda<\hat{\lambda}_1$, then problem \eqref{eqp} admits a unique solution $\bar{u}_{\lambda}\in D_+$.
\end{prop}
\begin{proof}
	By Theorem \ref{th19} we already have a positive solution $\bar{u}_{\lambda}\in D_+$. Suppose that $\bar{u}_{\lambda}$ is another positive solution of \eqref{eqp}. Again we have that $\bar{v}_{\lambda}\in D_+$. By Proposition 2.1 of Marano \& Papageorgiou \cite{15}, we can find $t>0$ such that
	\begin{equation}\label{eq38}
		t\bar{v}_{\lambda}\leq\bar{u}_{\lambda}.
	\end{equation}
	
	Let $t>0$ be the biggest real for which (\ref{eq38}) holds. Suppose that $t<1$. Also, let $\rho=||\bar{u}_{\lambda}||_{\infty}$ and let $\hat{\xi}_{\rho}>0$ be as postulated by hypothesis $H(f)_1^{''}(iv)$. We can always assume that $\hat{\xi}_{\rho}>||\xi||_{\infty}$. Also let $\bar{m}_{\lambda}=\min\limits_{\overline{\Omega}}\bar{v}_{\lambda}>0$. We have
	\begin{eqnarray*}
		&&-\Delta_p(t\bar{v}_{\lambda})+[\xi(z)+\hat{\xi}_{\rho}](t\bar{v}_{\lambda})^{p-1}\\
		&=&t^{p-1}(-\Delta_p\bar{v}_{\lambda}+[\xi(z)+\hat{\xi}_{\rho}]\bar{v}_{\lambda}^{p-1})\\
		&=&t^{p-1}(\lambda\bar{v}_{\lambda}^{p-1}+f(z,\bar{v}_{\lambda})+\hat{\xi}_{\rho}\bar{v}_{\lambda}^{p-1})\ (\mbox{since}\ \bar{v}_{\lambda}\in S_{\lambda})\\
		&\leq&\lambda(t\bar{v}_{\lambda})^{p-1}+f(z,t\bar{v}_{\lambda})+\hat{\xi}_{\rho}(t\bar{v}_{\lambda})^{p-1}-(1-t)\bar{m}_{\lambda}^{p-1}\\
		&&(\mbox{see hypothesis}\ H(f)_1^{''}(v)\ \mbox{and recall that}\ t<1)\\
		&<&\lambda\bar{u}_{\lambda}^{p-1}+f(z,\bar{u}_{\lambda})+\hat{\xi}_{\rho}\bar{u}_{\lambda}^{p-1}\ (\mbox{see (\ref{eq38}) and hypothesis}\ H(f)_1^{''}(iv))\\
		&=&-\Delta_p\bar{u}_{\lambda}+[\xi(z)+\hat{\xi}_{\rho}]\bar{u}_{\lambda}^{p-1}\ \mbox{for almost all}\ z\in\Omega\ (\mbox{since}\ \bar{u}_{\lambda}\in S_{\lambda}),\\
		\Rightarrow&&\bar{u}_{\lambda}-t\bar{v}_{\lambda}\in {\rm int}\, C_+\ (\mbox{see Proposition \ref{prop5}}),
	\end{eqnarray*}
	which contradicts the maximality of $t>0$. Therefore $t\geq 1$ and we have
	$$\bar{v}_{\lambda}\leq\bar{u}_{\lambda}\ (\mbox{see (\ref{eq38})}).$$
	
	Interchanging the roles of $\bar{u}_{\lambda}$ and $\bar{v}_{\lambda}$ in the above argument, we obtain
	\begin{eqnarray*}
		&&\bar{u}_{\lambda}\leq\bar{v}_{\lambda},\\
		&\Rightarrow&\bar{u}_{\lambda}=\bar{v}_{\lambda}.
	\end{eqnarray*}
	
	This proves the uniqueness of the positive solution of problem \eqref{eql}.
\end{proof}

So, we can state the following existence and uniqueness theorem for problem \eqref{eql}.
\begin{theorem}\label{th21}
	If hypotheses $H(\xi),H(\beta),H(f)_1^{''}$ hold, then
	\begin{itemize}
		\item[(a)] for every $\lambda\geq\hat{\lambda}_1=\hat{\lambda}_1(p,\xi,\beta)$, problem \eqref{eql} has no positive solutions;
		\item[(b)] for every $\lambda<\hat{\lambda}_1$, problem \eqref{eql} has a unique positive solution $\bar{u}_{\lambda}\in D_+$ and the map $\lambda\mapsto\bar{u}_{\lambda}$ from $(-\infty,\hat{\lambda}_1)$ into $C^1(\overline{\Omega})$ is
		\begin{itemize}
			\item[$\bullet$] strictly increasing (that is, $\vartheta<\lambda<\hat{\lambda}_1\Rightarrow\bar{u}_{\lambda}-\bar{u}_{\vartheta}\in {\rm int}\, C_+$);
			\item[$\bullet$] left continuous.
		\end{itemize}
	\end{itemize}
\end{theorem}

For the general nonhomogeneous problem, to have uniqueness, we need to set $\xi\equiv 0$. So, we consider the problem:
\begin{equation}\tag{$P_{\lambda}'$}\label{eqp'}
	\left\{\begin{array}{l}
		-{\rm div}\,a(Du(z))=\lambda u(z)^{p-1}+f(z,u(z))\ \mbox{in}\ \Omega,\\
		\frac{\partial u}{\partial n_a}+\beta(z)u^{p-1}=0\ \mbox{on}\ \partial\Omega,u>0,\lambda\in\RR,1<p<\infty.
	\end{array}\right\}
\end{equation}

Then reasoning as in the proof of Proposition \ref{prop11} (see also Fragnelli, Mugnai \& Papageorgiou \cite[Theorem 7]{7}), we have uniqueness of the positive solution and we can formulate the following theorem.
\begin{theorem}\label{th22}
	If hypotheses $H(a),H(\beta),H(f)_1^{''}$ hold, then there exists $\lambda^*\in\RR$ such that
	\begin{itemize}
		\item[(a)] for every $\lambda\geq\lambda^*$, problem \eqref{eqp'} has no positive solutions;
		\item[(b)] for every $\lambda<\lambda^*$, problem \eqref{eqp'} has a unique positive solution $\bar{u}_{\lambda}\in D_+$ and the map $\lambda\mapsto\bar{u}_{\lambda}$ from $(-\infty,\lambda^*)$ into $C^1(\overline{\Omega})$ is
		\begin{itemize}
			\item[$\bullet$] strictly increasing (that is, $\vartheta<\lambda<\lambda^*\Rightarrow\bar{u}_{\lambda}-\bar{u}_{\vartheta}\in {\rm int}\, C_+$);
			\item[$\bullet$] left continuous;
		\end{itemize}
	\end{itemize}
\end{theorem}

\section{$(p-1)$-superlinear perturbation}

In this section we examine what happens in problem \eqref{eqp} when the perturbation $f(z,\cdot)$ is $(p-1)$-superlinear. We do not assume that $f(z,\cdot)$ satisfies the usual (for ``superlinear'' problems) ``Ambrosetti-Rabinowitz condition''  (the ``AR-condition'' for short). Instead, we employ a less restrictive condition involving the function
$$d(z,x)=f(z,x)x-pF(z,x)\ \mbox{for all}\ (z,x)\in\Omega\times\RR.$$

In this way we incorporate in our framework $(p-1)$-superlinear functions with ``slower'' growth near $+\infty$, which fail to satisfy the AR-condition.

So, we introduce the following condition on the perturbation $f(z,x)$.

\smallskip
$H(f)_2:$ $f:\Omega\times\RR\rightarrow\RR$ is a Carath\'eodory function such that $f(z,0)=0$ for almost all $z\in\Omega$ and
\begin{itemize}
	\item[(i)] $|f(z,x)|\leq a(z)(1+x^{r-1})$ for almost all $z\in\Omega$, and for all $x\geq 0$, with $a\in L^{\infty}(\Omega)$, $p<r<p^*$;
	\item[(ii)] if $F(z,x)=\int^x_0f(z,s)ds$, then $\lim\limits_{x\rightarrow+\infty}\frac{F(z,x)}{x^p}=+\infty$ uniformly for almost all $z\in\Omega$;
	\item[(iii)] if $d(z,x)=f(z,x)x-pF(z,x)$, then $d(z,x)\leq d(z,y)+\nu(z)$ for almost all $z\in\Omega$, and for all $0\leq x\leq y$ with $\nu(\cdot)\in L^1(\Omega)$;
	\item[(iv)] $\lim\limits_{x\rightarrow 0^+}\frac{f(z,x)}{x^{\tau-1}}=0$ uniformly for almost all $z\in\Omega$, with $1<\tau<q$, $q\leq p$ as in $H(a)(iv)$ and there exist $s\in(\tau,q)$, $\delta_0>0$ such that $\tilde{c}_0x^{s-1}\leq f(z,x)$ for almost all $z\in\Omega$, $x\in[0,\delta_0]$ with $\tilde{c}_0>0$;
	\item[(v)] for every $\rho>0$, there exists $\hat{\xi}_{\rho}>0$ such that for almost all $z\in\Omega$ the mapping $x\mapsto f(z,x)+\hat{\xi}_{\rho}x^{p-1}$ is nondecreasing on $[0,\rho]$.
\end{itemize}
\begin{remark}
	Since we are looking for positive solutions and the above hypotheses concern the positive semiaxis $\RR_+=\left[0,+\infty\right)$, we may assume without any loss of generality, as we did in the sublinear case,  that $f(z,x)=0$ for almost all $z\in\Omega$, all $x\leq 0$. Hypotheses $H(f)_2(ii),(iii)$ imply that
	\begin{equation}\label{eq39}
		\lim\limits_{x\rightarrow+\infty}\frac{f(z,x)}{x^{p-1}}=+\infty\ \mbox{uniformly for almost all}\ z\in\Omega.
	\end{equation}
\end{remark}
	
	So, the perturbation $f(z,\cdot)$ is $(p-1)$-superlinear. Usually for such problems, the superlinerity is expressed through the AR-condition, which says that there exist $\tau>p$ and $M>0$ such that
	\begin{equation}\label{eq40}
		0<\tau F(z,x)\leq f(z,x)x\ \mbox{for almost all}\ z\in\Omega,\ \mbox{and for all}\ x\geq M\ \mbox{where}\ {\rm ess}\inf\limits_{\Omega}F(\cdot,M)>0.
	\end{equation}

	Here we have a unilateral version of the AR-condition, since $f(z,\cdot)_{\left(-\infty,0\right]}=0$. Integrating (\ref{eq39}) we obtain the more general condition
	\begin{equation}\label{eq41}
		c_{13}x^{\tau}\leq F(z,x)\ \mbox{for almost all}\ z\in\Omega,\ \mbox{for all}\ x\geq M,\ \mbox{and for some}\ c_{13}>0.
	\end{equation}

Evidently, (\ref{eq40}) and (\ref{eq41}) imply that (\ref{eq39}) holds. Using the AR-condition (\ref{eq40}) we can easily verify the C-condition for the energy functional. However, from (\ref{eq41}) we see that the AR-condition is rather restrictive. It excludes from consideration superlinear functions with slower growth near $+\infty$ (see the examples below). We have replaced the AR-condition by hypotheses $H(f)_2(ii),(iii)$, which incorporate in our framework such functions. Hypothesis $H(f)_2(iii)$ is a quasi-monotonicity condition on $d(z,\cdot)$ on $\RR_+$. This hypothesis is a slightly more general version of a condition used by Li \& Yang \cite{13}, who compared this condition with other superlinearity conditions that can be found in the literature.

\smallskip
{\it Examples.}
	The following functions satisfy hypotheses $H(f)_2$. For the sake of simplicity we drop the $z$-dependence.
	$$f_1(x)=\left\{\begin{array}{ll}
		0&\mbox{if}\ x<0\\
		x^{\tau-1}-2x^{\vartheta-1}&\mbox{if}\ 0\leq x\leq 1\\
		x^{r-1}-2x^{p-1}&\mbox{if}\ 1<x
	\end{array}\right.$$
	with $1<\tau<\vartheta<p<r$
	$$f_2(x)=\left\{\begin{array}{ll}
		0&\mbox{if}\ x<0\\
		x^{\tau-1}-2x^{\vartheta-1}&\mbox{if}\ 0\leq x\leq 1\\
		x^{p-1}(\ln x-1)&\mbox{if}\ 1<x
	\end{array}\right.$$
	with $1<\tau<\vartheta<p$. Note that $f_1$ satisfies the AR-condition, whereas $f_2$ does not. Also, both functions may be sign-changing.
	
	As before, we denote
	\begin{center}
		$\mathcal{L}=\{\lambda>0:\mbox{problem \eqref{eqp} has a positive solution}\}$,\\
		$S_{\lambda}=$ the set of all positive solutions of problem \eqref{eqp}.
	\end{center}

\begin{prop}\label{prop23}
	If hypotheses $H(a),H(\xi),H(\beta),H(f)_2$ hold, then $\mathcal{L}\neq\emptyset$ and $S_{\lambda}\subseteq D_+$.
\end{prop}
\begin{proof}
	Let $\eta>||\xi||_{\infty}$ and consider the functional $\psi_{\lambda}:W^{1,p}(\Omega)\rightarrow\RR$ defined by
	$$\psi_{\lambda}(u)=\frac{1}{p}\mu(u)+\frac{\eta}{p}||u^-||^p_p-\frac{\lambda}{p}||u^+||^p_p-\int_{\Omega}F(z,u^+)dz\ \mbox{for all}\ u\in W^{1,p}(\Omega).$$
	
	Hypotheses $H(f)_2(i),(iv)$ imply that given $\epsilon>0$, we can find $c_{14}=c_{14}(\epsilon)>0$ such that
	\begin{equation}\label{eq42}
		F(z,x)\leq\frac{\epsilon}{p}x^{\tau}+c_{14}x^r\ \mbox{for almost all}\ z\in\Omega,\ \mbox{and for all}\ x\geq 0.
	\end{equation}
	
	Then for $\lambda<0$, with $|\lambda|>||\xi||_{\infty}$, we have
	\begin{eqnarray}\label{eq43}
		\psi_{\lambda}(u)&\geq&\frac{1}{p}\mu(u^-)+\frac{\eta}{p}||u^-||^p_p+\frac{1}{p}\mu(u^+)+\frac{|\lambda|}{p}||u^+||^p_p-\epsilon c_{15}||u||^{\tau}-c_{16}||u||^{\tau}\nonumber\\
		&&\mbox{with}\ c_{15},c_{16}>0\ (\mbox{see (\ref{eq42})})\nonumber\\
		&\geq&\left[c_{17}-(\epsilon c_{15}||u||^{\tau-p}+c_{16}||u||^{r-p})\right]||u||^p\ \mbox{for some}\ c_{17}>0.
	\end{eqnarray}
	
	Let $k_0(t)=\epsilon c_{15}t^{\tau-p}+c_{16}t^{r-p}$ $t\geq 0$. Since $1<\tau<p<r$, we see that $k_0(t)\rightarrow+\infty$ as $t\rightarrow 0^+$ and as $t\rightarrow+\infty$. So, we can find $t_0>0$ such that $k_0(t_0)=\min\limits_{t>0}k_0$. We have
	\begin{eqnarray*}
		&&k'_0(t_0)=0,\\
		&\Rightarrow&t_0=\left[\frac{\epsilon c_{15}(p-\tau)}{c_{16}(r-p)}\right]^{\frac{1}{r-\tau}}\,.
	\end{eqnarray*}
	
	Then $k_0(t_0)\rightarrow 0^+$ as $\epsilon\rightarrow 0^+$. So,  it follows from (\ref{eq43}) that we can find small enough $\rho\in(0,1)$  such that
	\begin{equation}\label{eq44}
		0<\inf\{\psi_{\lambda}(u):||u||=\rho\}=m^{\lambda}_{\rho}.
	\end{equation}
	
	Hypothesis $H(f)_2(ii)$ implies that if $u\in D_+$, then
	\begin{equation}\label{eq45}
		\psi_{\lambda}(tu)\rightarrow-\infty\ \mbox{as}\ t\rightarrow+\infty.
	\end{equation}
	\begin{claim}
		For every $\lambda\in\RR,\psi_{\lambda}(\cdot)$ satisfies the C-condition.
	\end{claim}
	
	Consider a sequence $\{u_n\}_{n\geq 1}\subseteq W^{1,p}(\Omega)$ such that
	\begin{eqnarray}
		&&|\psi_{\lambda}(u_n)|\leq M_1\ \mbox{for some}\ M_1>0,\ \mbox{and for all}\ n\in\NN,\label{eq46}\\
		&&(1+||u_n||)\psi'_{\lambda}(u_n)\rightarrow 0\ \mbox{in}\ W^{1,p}(\Omega)^*\ \mbox{as}\ n\rightarrow\infty.\label{eq47}
	\end{eqnarray}
	
	From (\ref{eq47}) we have
	\begin{eqnarray}\label{eq48}
		&&|\left\langle A(u_n),h\right\rangle+\int_{\Omega}\xi(z)|u_n|^{p-2}u_nhdz+\int_{\partial\Omega}\beta(z)|u_n|^{p-2}hd\sigma-\eta\int_{\Omega}(u_n^-)^{p-1}hdz\nonumber\\
		&&\hspace{1cm}-\int_{\Omega}\lambda(u_n^+)^{p-1}hdz-\int_{\Omega}f(z,u_n^+)hdz|\leq\frac{\epsilon_n||h||}{1+||u_n||}\\
		&&\mbox{for all}\ h\in W^{1,p}(\Omega),\ \mbox{with}\ \epsilon_n\rightarrow 0^+\nonumber.
	\end{eqnarray}
	
	In (\ref{eq48}) we choose $h=-u_n^-\in W^{1,p}(\Omega)$. Then
	\begin{eqnarray}\label{eq49}
			&&\mu(u_n^-)+\eta||u_n^-||^p_p\leq\epsilon_n\ \mbox{for all}\ n\in\NN,\nonumber\\
			&\Rightarrow&c_{18}||u_n^-||^p\leq \epsilon_n\ \mbox{for some}\ c_{18}>0,\ \mbox{and for all}\ n\in\NN\ (\mbox{recall that}\ \eta>||\xi||_{\infty}),\nonumber\\
			&\Rightarrow&u_n^-\rightarrow 0\ \mbox{in}\ W^{1,p}(\Omega)\ \mbox{as}\ n\rightarrow\infty.
	\end{eqnarray}
	
	Next, in (\ref{eq48}) we choose $h=u_n^+\in W^{1,p}(\Omega)$. Then
	\begin{equation}\label{eq50}
		-\mu(u_n^+)+\lambda||u_n^+||^p_p+\int_{\Omega}f(z,u_n^+)u_n^+dz\leq\epsilon_n\ \mbox{for all}\ n\in\NN.
	\end{equation}
	
	From (\ref{eq46}) and (\ref{eq49}) we have
	\begin{equation}\label{eq51}
		\mu(u_n^+)-\lambda||u_n^+||^p_p-\int_{\Omega}pF(z,u_n^+)dz\leq M_2\ \mbox{for some}\ M_2>0,\ \mbox{and for all}\ n\in\NN.
	\end{equation}
	
	We add (\ref{eq50}), (\ref{eq51}) and obtain
	\begin{equation}\label{eq52}
		\int_{\Omega}d(z,u_n^+)dz\leq M_3\ \mbox{for some}\ M_3>0,\ \mbox{and for all}\ n\in\NN.
	\end{equation}
	
	We will use (\ref{eq52}) to show that $\{u_n^+\}_{n\geq 1}\subseteq W^{1,p}(\Omega)$ is bounded. Arguing by contradiction, suppose that $||u_n^+||\rightarrow\infty$. We set $y_n=\frac{u_n^+}{||u_n^+||}$ for all $n\in\NN$. We have $||y_n||=1$ for all $n\in\NN$ and so we may assume that
	\begin{equation}\label{eq53}
		y_n\stackrel{w}{\rightarrow}y\ \mbox{in}\ W^{1,p}(\Omega)\ \mbox{and}\ y_n\rightarrow y\ \mbox{in}\ L^p(\Omega)\ \mbox{and in}\ L^p(\partial\Omega)\ \mbox{as}\ n\rightarrow\infty.
	\end{equation}
	
	First, we assume that $y\neq 0$. Let $\Omega_0=\{z\in\Omega:y(z)=0\}$. Then $|\Omega\backslash\Omega_0|_N>0$ (by $|\cdot|_N$ we denote the Lebesgue measure on $\RR^N$) and $u_n^+(z)\rightarrow+\infty$ for almost all $z\in\Omega\backslash\Omega_0$ as $n\rightarrow\infty$. Hence hypothesis $H(f)_2(ii)$ implies that
	\begin{eqnarray}\label{eq54}
		&&\frac{F(z,u_n^+(z))}{||u_n^+||^p}=\frac{F(z,u_n^+(z))}{u_n^+(z)^p}y_n(z)^p\rightarrow+\infty\ \mbox{for almost all}\ z\in\Omega\backslash\Omega_0\ \mbox{as}\ n\rightarrow\infty,\nonumber\\
		&\Rightarrow&\int_{\Omega}\frac{F(z,u_n^+)}{||u_n^+||^p}dz\rightarrow+\infty\ \mbox{as}\ n\rightarrow+\infty\ (\mbox{by Fatou's lemma}).
	\end{eqnarray}
	
	Corollary \ref{cor3} and hypothesis $H(a)(iv)$ imply that
	\begin{equation}\label{eq55}
		G(y)\leq c_{19}(|y|^q+|y|^p)\ \mbox{for some}\ c_{19}>0,\ \mbox{and for all}\ y\in\RR^N.
	\end{equation}
	
	From (\ref{eq46}) and (\ref{eq49}), we have
	\begin{eqnarray}\label{eq56}
		&&-\int_{\Omega}G(Du_n^+)dz-\frac{1}{p}\int_{\Omega}\xi(z)(u_n^+)^pdz-\frac{1}{p}\int_{\partial\Omega}\beta(z)(u_n^+)^pd\sigma+\frac{\lambda}{p}||u_n^+||^p_p+\nonumber\\
		&&\int_{\Omega}F(z,u_n^+)dz\leq M_4\ \mbox{for some}\ M_4>0,\ \mbox{and for all}\ n\in\NN\nonumber\\
		&\Rightarrow&\int_{\Omega}\frac{F(z,u_n^+)}{||u_n^+||^p}dz\leq M_5\ \mbox{for some}\ M_5>0,\ \mbox{and for all}\ n\in\NN\\
		&&(\mbox{see (\ref{eq53}), (\ref{eq55}) and hypotheses}\ H(\xi),H(\beta)).\nonumber
	\end{eqnarray}
	
	Comparing (\ref{eq54}) and (\ref{eq56}), we get a contradiction.
	
	So, we assume that $y=0$. We consider the $C^1$-functional $\hat{\psi}_{\lambda}:W^{1,p}(\Omega)\rightarrow\RR$ defined by
	\begin{eqnarray*}
		&&\hat{\psi}_{\lambda}(u)=\frac{c_1}{p(p-1)}||Du||^p_p+\frac{1}{p}\int_{\Omega}\xi(z)|u|^pdz+\frac{1}{p}\int_{\partial\Omega}\beta(z)|u|^pd\sigma+\frac{\eta}{p}||u^-||^p_p\\
		&&-\frac{\lambda}{p}||u^+||^p_p-\int_{\Omega}F(z,u^+)dz\ \mbox{for all}\ u\in W^{1,p}(\Omega).
	\end{eqnarray*}
	
	Evidently, $\hat{\psi}_{\lambda}\leq\psi_{\lambda}$ (see Corollary \ref{cor3}).
	
	We define $\vartheta_n(t)=\hat{\psi}_{\lambda}(tu_n^+)$ for all $t\in[0,1]$, and for all $n\in\NN$. Let $t_n\in[0,1]$ be such that
	\begin{equation}\label{eq57}
		\vartheta_n(t_n)=\max\limits_{0\leq t\leq 1}\vartheta_n(t)=\max\limits_{0\leq t\leq 1}\hat{\psi}_{\lambda}(t u_n^+)\ \mbox{for all}\ n\in\NN.
	\end{equation}
	
	For $\gamma>0$, let $v_n=(2\gamma)^{1/p}y_n\in W^{1,p}(\Omega)$. Evidently, $v_n\rightarrow 0$ in $L^r(\Omega)$ (see (\ref{eq53}) and recall that we have assumed that $y=0$). Then
	\begin{equation}\label{eq58}
		\int_{\Omega}F(z,v_n)dz\rightarrow 0\ \mbox{as}\ n\rightarrow\infty.
	\end{equation}
	
	Since $||u_n^+||\rightarrow\infty$, we can find $n_0\in\NN$ such that
	\begin{equation}\label{eq59}
		(2\gamma)^{1/p}\frac{1}{||u_n^+||}\in(0,1)\ \mbox{for all}\ n\geq n_0.
	\end{equation}
	
	Then (\ref{eq57}) and (\ref{eq59}) imply that
	\begin{eqnarray}\label{eq60}
		\vartheta_n(t_n)&\geq&\vartheta_n\left(\frac{(2\gamma)^{1/p}}{||u_n^+||}\right)\ \mbox{for all}\ n\geq n_0\nonumber\\
		\Rightarrow\hat{\psi}_{\lambda}(t_nu_n^+)&\geq&\hat{\psi}_{\lambda}((2\gamma)^{1/p}y_n)=\hat{\psi}_{\lambda}(v_n)\ \mbox{for all}\ n\geq n_0\nonumber\\
		&\geq&\frac{2\gamma c_1}{p(p-1)}\left(||Dy_n||^p_p+\frac{p-1}{c_1}\int_{\Omega}[\xi(z)+\eta-\lambda]y_n^pdz\right)\nonumber\\
		&&-\left(\int_{\Omega}F(z,v_n)dz+\frac{\eta}{p}||v_n||^p_p\right)\nonumber\\
		&\geq&\frac{2\gamma c_{20}}{p(p-1)}-\left(\int_{\Omega}F(z,v_n)dz+\frac{\eta}{p}||v_n||^p_p\right)\ \mbox{for some}\ c_{20}>0,\ \mbox{and for all}\ n\geq n_0.
	\end{eqnarray}
	
	Recall that $v_n\rightarrow 0$ in $L^p(\Omega)$. Using this fact and (\ref{eq58}) in (\ref{eq60}), we see that
	$$\hat{\psi}_{\lambda}(t_nu_n^+)\geq\frac{\gamma c_{20}}{p(p-1)}\ \mbox{for some}\ n\geq n_1\geq n_0.$$
	
	However,
	recall that $\gamma>0$ is arbitrary. So, it follows that
	\begin{equation}\label{eq61}
		\hat{\psi}_{\lambda}(t_nu_n^+)\rightarrow+\infty\ \mbox{as}\ n\rightarrow\infty.
	\end{equation}
	
	We have $0\leq t_nu_n^+\leq u_n^+$ for all $n\in\NN$. So, on account of hypothesis $H(f)_2(iii)$, we have
	\begin{eqnarray}\label{eq62}
		&&\int_{\Omega}d(z,t_nu_n^+)dz\leq\int_{\Omega}d(z,u_n^+)dz+||\nu||_1\leq M_6\\
		&&\mbox{for some}\ M_6>0,\ \mbox{and for all}\ n\in\NN\ (\mbox{see (\ref{eq52})}).\nonumber
	\end{eqnarray}
	
	We know that
	\begin{eqnarray}\label{eq63}
		&&\hat{\psi}_{\lambda}(0)=0\ \mbox{and}\ \hat{\psi}_{\lambda}(u_n^+)\leq M_7\ \mbox{for some}\ M_7>0\\
		&&(\mbox{see (\ref{eq46}), (\ref{eq52}) and recall that}\ \hat{\psi}_\lambda\leq\psi_\lambda).\nonumber
	\end{eqnarray}
	
	From (\ref{eq61}) and (\ref{eq63}) we infer that $t_n\in(0,1)$ for all $n\geq n_2$. Hence we have
	\begin{equation}\label{eq64}
		0=t_n\frac{d}{dt}\hat{\psi}_\lambda(tu_n^+)|_{t=t_n}=\left\langle \psi'_{\lambda}(t_nu_n^+),t_nu_n^+\right\rangle\ \mbox{for all}\ n\geq n_2\ (\mbox{see (\ref{eq57})}).
	\end{equation}
	
	Combining (\ref{eq62}) and (\ref{eq64}) we see that
	\begin{equation}\label{eq65}
		p\hat{\psi}_\lambda(t_nu_n^+)\leq M_6\ \mbox{for all}\ n\geq n_2.
	\end{equation}
	
	Comparing (\ref{eq61}) and (\ref{eq65}) we have a contradiction. Therefore
	\begin{eqnarray*}
		&&\{u_n^+\}_{n\geq 1}\subseteq W^{1,p}(\Omega)\ \mbox{is bounded},\\
		&\Rightarrow&\{u_n\}_{n\geq 1}\subseteq W^{1,p}(\Omega)\ \mbox{is bounded (see (\ref{eq49}))}.
	\end{eqnarray*}
	
	So, we may assume that
	\begin{equation}\label{eq66}
		u_n\stackrel{w}{\rightarrow}u\ \mbox{in}\ W^{1,p}(\Omega)\ \mbox{and}\ u_n\rightarrow u\ \mbox{in}\ L^r(\Omega)\ \mbox{and in}\ L^p(\partial\Omega).
	\end{equation}
	
	We return to (\ref{eq48}) and choose $h=u_n-u\in W^{1,p}(\Omega)$, pass to the limit as $n\rightarrow\infty$ and use (\ref{eq66}). Then
	\begin{eqnarray*}
		&&\lim\limits_{n\rightarrow\infty}\left\langle A(u_n),u_n-u\right\rangle=0,\\
		&\Rightarrow&u_n\rightarrow u\ \mbox{in}\ W^{1,p}(\Omega)\ (\mbox{see Proposition \ref{prop4}}).
	\end{eqnarray*}
	
	Therefore $\psi_\lambda$ satisfies the C-condition and this proves the claim.
	
	Then (\ref{eq44}), (\ref{eq45}) and the claim, permit the use of Theorem \ref{th1} (the mountain pass theorem). So, we can find $u_\lambda\in W^{1,p}(\Omega)$ $\left(\lambda<0,|\lambda|>||\xi||_{\infty}\right)$ such that
	\begin{equation}\label{eq67}
		u_\lambda\in K_{\psi_\lambda}\ \mbox{and}\ m^{\lambda}_{\rho}\leq\psi_\lambda(u_\lambda).
	\end{equation}
	
It follows	from (\ref{eq67})  that $u_\lambda\neq 0$ (see (\ref{eq44})) and
	\begin{eqnarray}\label{eq68}
		&&\left\langle A(u_\lambda),h\right\rangle+\int_{\Omega}\xi(z)|u_\lambda|^{p-2}u_\lambda hdz+\int_{\partial\Omega}\beta(z)|u_\lambda|^{p-2}u_\lambda hd\sigma\nonumber\\
		&&-\eta\int_{\Omega}(u^-_{\lambda})^{p-1}hdz=\lambda\int_{\Omega}(u^+_{\lambda})^{p-1}hdz+\int_{\Omega}f(z,u_\lambda^+)hdz\ \mbox{for all}\ h\in W^{1,p}(\Omega).
	\end{eqnarray}
	
	In (\ref{eq68}) we choose $h=-u_{\lambda}^-\in W^{1,p}(\Omega)$. Then
	\begin{eqnarray*}
		&&\frac{c_1}{p-1}||Du^-_{\lambda}||^p_p+\int_{\Omega}[\xi(z)+\eta]|u_\lambda^-|^pdz\leq 0\ (\mbox{see Lemma \ref{lem2}}),\\
		&\Rightarrow&u_\lambda\geq 0,\ u_\lambda\neq 0.
	\end{eqnarray*}
	
	It follows from (\ref{eq68}) that $u_\lambda$ is a positive solution of \eqref{eqp}, hence $\lambda\in\mathcal{L}$ and so $\mathcal{L}\neq\emptyset$. Moreover, from the nonlinear regularity theory (see \cite{12}) and the nonlinear maximum principle (see \cite{24}), we can deduce that $S_\lambda\subseteq D_+$.
\end{proof}

In the present setting, on account of hypotheses $H(f)_2(i),(iv)$, we have that
\begin{equation}\label{eq69}
	\lambda x^{p-1}+f(z,x)\geq \tilde{c}_0x^{s-1}-c_{21}x^{r-1}\ \mbox{for almost all}\ z\in\Omega,\ \mbox{and for all}\ x\geq 0,
\end{equation}
for some big enough $c_{21}=c_{21}(\lambda)>0$. An inspection of the proofs of Propositions \ref{prop9}$-$\ref{prop14} reveals that their conclusions remain valid in the present setting. Now, instead of (\ref{eq19}) we use (\ref{eq69}). So, we can state the following proposition summarizing these conclusions.
\begin{prop}\label{prop24}
	If hypotheses $H(a),H(\xi),H(\beta),H(f)_2$ hold, then
	\begin{itemize}
		\item[(a)] if $\lambda\in\mathcal{L}$ and $\vartheta<\lambda$, then $\vartheta\in\mathcal{L}$;
		\item[(b)] $\lambda^*=\sup\mathcal{L}<+\infty$;
		\item[(c)] for every $\lambda\in\mathcal{L}$, problem \eqref{eqp} admits a smallest element $\bar{u}_\lambda\in D_+$ and the map $\lambda\mapsto\bar{u}_\lambda$ from $\mathcal{L}$ into $C^1(\overline{\Omega})$ is
		\begin{itemize}
			\item[$\bullet$] strictly increasing (that is, $\vartheta<\lambda\in\mathcal{L}\Rightarrow\bar{u}_\lambda-\bar{u}_\vartheta\in {\rm int}\, C_+$);
			\item[$\bullet$] left continuous.
		\end{itemize}
	\end{itemize}
\end{prop}

Again we show that the critical parameter $\lambda^*$ is not admissible, hence
$$\mathcal{L}=(-\infty,\lambda^*).$$
\begin{prop}\label{prop25}
	If hypotheses $H(a),H(\xi),H(\beta),H(f)_2$ hold, then $\lambda^*\notin\mathcal{L}$.
\end{prop}
\begin{proof}
	Arguing by contradiction, suppose that $\lambda^*\in\mathcal{L}$. Then according to Proposition \ref{prop24} problem \eqref{eqp} admits a smallest positive solution $\bar{u}_*=\bar{u}_{\lambda^*}\in D_+$.

Consider $\lambda>\lambda^*$ and, as always, let $\eta>||\xi||_{\infty}$. We introduce the Carath\'eodory function $\hat{\gamma}_\lambda(z,x)$ define by
	\begin{eqnarray}\label{eq70}
		\hat{\gamma}_\lambda(z,x)=\left\{\begin{array}{lll}
			&(\lambda+\eta)\bar{u}_*(z)^{p-1}+f(z,\bar{u}_*(z))&\mbox{if}\ x\leq\bar{u}_*(z)\\
			&(\lambda+\eta)x^{p-1}+f(z,x)&\mbox{if}\ \bar{u}_*(z)<x.
		\end{array}\right.
	\end{eqnarray}
	
	We set $\hat{\Gamma}_\lambda(z,x)=\int^x_0\hat{\gamma}_\lambda(z,s)ds$ and consider the $C^1$-functional $\tilde{\psi}_\lambda:W^{1,p}(\Omega)\rightarrow\RR$ defined by
	$$\tilde{\psi}_\lambda(u)=\frac{1}{p}\mu(u)+\frac{\eta}{p}||u|^p_p-\int_{\Omega}\Gamma_\lambda(z,u)dz\ \mbox{for all}\ u\in W^{1,p}(\Omega).$$
	
	Recall that $\bar{u}_*\in D_+$ hence $\bar{m}_*=\min\limits_{\overline{\Omega}}\bar{u}_*>0$. Let $\rho_0\in(0,\bar{m}_*)$. Then for $u\in C^1(\overline{\Omega})$ with $||u||_{C^1(\overline{\Omega})}\leq\rho_0$, we have
	\begin{eqnarray}\label{eq71}
		&&\tilde{\psi}_\lambda(u)=\tilde{\psi}_\lambda(0)\ (\mbox{see (\ref{eq70})})\nonumber\\
		&\Rightarrow&u=0\ \mbox{is a local}\ C^1(\overline{\Omega})\mbox{-minimizer of}\ \tilde{\psi}_\lambda,\nonumber\\
		&\Rightarrow&u=0\ \mbox{is a local}\ W^{1,p}(\Omega)\mbox{-minimizer of}\ \tilde{\psi}_\lambda\ \mbox{(see Proposition \ref{prop6})}.
	\end{eqnarray}
	
	Using (\ref{eq70}) we can easily verify that
	\begin{equation}\label{eq72}
		K_{\tilde{\psi}_{\lambda}}\subseteq\left[\bar{u}_*\right)\cap C^1(\overline{\Omega}).
	\end{equation}
	
	Without any loss of generality, we assume that $K_{\psi_\lambda}$ is finite (see (\ref{eq70}), (\ref{eq72})). Then on account of (\ref{eq71}) we can find small enough $\rho_1\in(0,1)$  such that
	\begin{equation}\label{eq73}		\inf\{\tilde{\psi}_\lambda(u):||u||=\rho_1\}=\tilde{m}_1>\tilde{\psi}_\lambda(0)=\tilde{\psi}_\lambda(\bar{u}_*)\ (\mbox{see \cite{1}}).
	\end{equation}
	
	Also, hypothesis $H(f)_2(ii)$ implies that if $u\in D_+$, then
	\begin{equation}\label{eq74}
		\tilde{\psi}_\lambda(tu)\rightarrow-\infty\ \mbox{as}\ t\rightarrow+\infty.
	\end{equation}
	
	Moreover, on account of (\ref{eq70}), reasoning as in the proof of Proposition \ref{prop23} (see the claim), we can show that
	\begin{equation}\label{eq75}
		\tilde{\psi}_\lambda(\cdot)\ \mbox{satisfies the C-condition}.
	\end{equation}
	
	Then (\ref{eq73}), (\ref{eq74}), (\ref{eq75}) permit the use of Theorem \ref{th1} (the mountain pass theorem). Hence we can find $u_\lambda\in W^{1,p}(\Omega)$ such that
	$$u_\lambda\in K_{\tilde{\psi}_\lambda}\subseteq\left[\bar{u}_*\right)\cap C^1(\overline{\Omega})\ (\mbox{see (\ref{eq72})}),\ \tilde{\psi}_\lambda(u_\lambda)\geq\tilde{m}_1.$$
	
	It follows that $u_\lambda\in S_\lambda$ and so $\lambda\in\mathcal{L}$, a contradiction. Therefore $\lambda^*\not\in\mathcal{L}$. The proof is now complete.
\end{proof}

{\rm In this case for $\lambda\in\mathcal{L}=(-\infty,\lambda^*)$ we have a multiplicity result for problem \eqref{eqp}. }

\begin{prop}\label{prop26}
	If hypotheses $H(a),H(\xi),H(\beta),H(f)_2$ hold and $\lambda\in\mathcal{L}=(-\infty,\lambda^*)$, then problem \eqref{eqp} admits at least two positive solutions
	$$u_\lambda,\ \hat{u}_\lambda\in D_+,\ u_\lambda\leq\hat{u}_\lambda,\ u_\lambda\neq\hat{u}_\lambda.$$
\end{prop}
\begin{proof}
	Since $\lambda\in\mathcal{L}$ we can find $u_\lambda\in S_\lambda\subseteq D_+$ (see Proposition \ref{prop23}). We may assume that $u_\lambda$ is the minimal positive solution of \eqref{eqp} produced in Proposition \ref{prop24} (that is, $u_\lambda=\bar{u}_\lambda$). With $\eta>||\xi||_{\infty}$, we introduce the Carath\'eodory function $k_\lambda(z,x)$ defined by
	\begin{eqnarray}\label{eq76}
		k_\lambda(z,x)=\left\{\begin{array}{lll}
			&(\lambda+\eta)u_\lambda(z)^{p-1}+f(z,u_\lambda(z))&\mbox{if}\ x\leq u_\lambda(z)\\
			&(\lambda+\eta)x^{p-1}+f(z,x)&\mbox{if}\ u_\lambda(z)<x.
		\end{array}\right.
	\end{eqnarray}
	
	Let $K_\lambda(z,x)=\int^x_0k_\lambda(z,s)ds$ and consider the $C^1$-functional $j_\lambda:W^{1,p}(\Omega)\rightarrow\RR$ defined by
	$$j_\lambda(u)=\frac{1}{p}\mu(u)+\frac{\eta}{p}||u||^p_p-\int_{\Omega}K_\lambda(z,u)dz.$$
	
	Working with $j_\lambda(\cdot)$ as in the proof of Proposition \ref{prop25} and using (\ref{eq76}), we produce $\hat{u}_\lambda\in W^{1,p}(\Omega)$ such that
	\begin{equation}\label{eq77}
		\hat{u}_\lambda\in K_{j_\lambda}\subseteq\left[u_\lambda\right)\cap C^1(\overline{\Omega}),\ \hat{u}_\lambda\notin\{0,u_\lambda\}.
	\end{equation}
	
It follows	from (\ref{eq76}) and (\ref{eq77}) that $\hat{u}_\lambda\in D_+$ is the second positive solution of \eqref{eqp}.
\end{proof}

{\rm Summarizing the situation for the ``superlinear'' case, we can state the following result.}
\begin{theorem}\label{th27}
	If hypotheses $H(a),H(\xi),H(\beta),H(f)_2$ hold, then there exists $\lambda^*<+\infty$ such that
	\begin{itemize}
		\item[(a)] for every $\lambda\geq\lambda^*$, problem \eqref{eqp} has no positive solutions;
		\item[(b)] for every $\lambda<\lambda^*$, problem \eqref{eqp} has at least two positive solutions $u_\lambda$, $\hat{u}_\lambda\in D_+$, $u_\lambda\leq\hat{u}_\lambda$, $u_\lambda\neq\hat{u}_\lambda$;
		\item[(c)] for every $\lambda<\lambda^*$, problem \eqref{eqp} has a smallest positive solution $\bar{u}_\lambda\in D_+$ and the map $\lambda\mapsto\bar{u}_\lambda$ from $\mathcal{L}=(-\infty,\hat{\lambda}_1)$ into $C^1(\overline{\Omega})$ is
		\begin{itemize}
			\item[$\bullet$] strictly increasing (that is, $\vartheta<\lambda\in\mathcal{L}\Rightarrow\bar{u}_\lambda-\bar{u}_\vartheta\in int C_+$);
			\item[$\bullet$] left continuous.
		\end{itemize}
	\end{itemize}
\end{theorem}

{\rm Again, in the special case of the $p$-Laplacian, see problem \eqref{eql} ($a(y)=|y|^{p-2}y$ for all $y\in\RR^N$), we can identify $\lambda^*$ as $\hat{\lambda}_1=\hat{\lambda}_1(p,\xi,\beta)$, provided that $f(z,x)>0$ for almost all $z\in\Omega$, and for all $x>0$ and we restrict the condition near zero (that is, $H(f)_2(iv)$).

So, the new conditions on the perturbation $f(z,x)$ are the following:

\smallskip
$H(f)'_2:$ $f:\Omega\times\RR\rightarrow\RR$ is a Carath\'eodory function such that for almost all $z\in\Omega$, $f(z,0)=0$, $f(z,x)>0$ for all $x>0$, hypotheses $H(f)'_2(i),(ii),(iii),(v)$ are the same as the corresponding hypotheses $H(f)'_2(i),(ii),(iii),(v)$ and
\begin{itemize}
	\item[(iv)] $\lim\limits_{x\rightarrow 0^+}\frac{f(z,x)}{x^{p-1}}=0$ uniformly for almost all $z\in\Omega$.
\end{itemize}

From Proposition \ref{prop17}, we already know that for $\lambda\geq\hat{\lambda}_1=\hat{\lambda}_1(p,\xi,\beta)$ problem \eqref{eql} has no positive solutions.}

\begin{prop}\label{prop28}
	If hypotheses $H(\xi),H(\beta),H(f)'_2$ hold and $\lambda<\hat{\lambda}_1$, then $\lambda\in\mathcal{L}$.
\end{prop}
\begin{proof}
	Let $\lambda\in(-\infty,\hat{\lambda}_1)$ and consider the Carath\'eodory function $\hat{\vartheta}_\lambda(z,x)$ defined by
	\begin{eqnarray}\label{eq78}
		\hat{\vartheta}_\lambda(z,x)=\left\{\begin{array}{lll}
			&0&\mbox{if}\ x\leq 0\\
			&\lambda x^{p-1}+f(z,x)&\mbox{if}\ 0<x.
		\end{array}\right.
	\end{eqnarray}
	
	We set $\hat{\Theta}_\lambda(z,x)=\int^x_0\hat{\vartheta}_\lambda(z,s)ds$ and with $\eta>||\xi||_{\infty}$, we consider the $C^1$-functional $w_\lambda:W^{1,p}(\Omega)\rightarrow\RR$ defined by
	$$w_\lambda(u)=\frac{1}{p}\mu(u)+\frac{\eta}{p}||u^-||^p_p-\int_{\Omega}\hat{\Theta}_\lambda(z,u)dz\ \mbox{for all}\ u\in W^{1,p}(\Omega).$$
	
	Hypotheses $H(f)'_2(i),(iv)$ imply that given $\epsilon>0$, we can find $c_{22}>0$ such that
	\begin{equation}\label{eq79}
		F(z,x)\leq\frac{\epsilon}{p}x^p+c_{22}x^r\ \mbox{for almost all}\ z\in\Omega,\ \mbox{and for all}\ x\geq 0.
	\end{equation}
	
	Then from (\ref{eq78}) and (\ref{eq79}), we have
	$$w_\lambda(u)\geq\frac{1}{p}[\mu(u^-)+\eta||u^-||^p_p]+\frac{1}{p}[\mu(u^+)-(\lambda+\epsilon)||u^+||^p_p]-c_{22}||u^+||^r_r\ (\mbox{see (\ref{eq79})}).$$
	
	Choosing $\epsilon\in(0,\hat{\lambda}_1-\lambda)$, we have
	\begin{eqnarray*}
		&&w_\lambda(u)\geq c_{23}||u||^p-c_{24}||u||^r\ \mbox{for some}\ c_{23},c_{24}>0,\\
		&\Rightarrow&u=0\ \mbox{local minimizer of}\ w_\lambda(\cdot)\ (\mbox{recall that}\ r>p).
	\end{eqnarray*}
	
	So, we can find $\rho\in(0,1)$ small such that
	\begin{equation}\label{eq80}
		w_\lambda(0)=0<\inf\{w_\lambda(u):||u||=\rho\}=m_\lambda
	\end{equation}
	(see Aizicovici, Papageorgiou \& Staicu \cite{1}, proof of Proposition 29).
	
	Also, hypothesis $H(f)'_2(ii)$ implies that if $u\in D_+$, then
	\begin{equation}\label{eq81}
		w_\lambda(tu)\rightarrow-\infty\ \mbox{as}\ t\rightarrow+\infty.
	\end{equation}
	
	Finally, from the proof of Proposition \ref{prop23} (see the claim), we know that
	\begin{equation}\label{eq82}
		w_\lambda(\cdot)\ \mbox{satisfies the C-condition}.
	\end{equation}
	
	Then (\ref{eq80}), (\ref{eq81}), (\ref{eq82}) permit the use of Theorem \ref{th1} (the mountain pass theorem) and produce $u_\lambda\in W^{1,p}(\Omega)$ such that
	\begin{eqnarray*}
		&&u_\lambda\in K_{w_\lambda}\subseteq D_+\cup\{0\}\ (\mbox{see the proof of Proposition \ref{prop8}})\\
		&&w_\lambda(0)=0<m_\lambda\leq w_\lambda(u_\lambda).
	\end{eqnarray*}
	
	Therefore $u_\lambda\in D_+$ is a positive solution of \eqref{eql}, hence $\lambda\in\mathcal{L}$.
\end{proof}

{\rm So, for problem \eqref{eql} we can state the following theorem covering the case of a $(p-1)$-superlinear perturbation.}

\begin{theorem}\label{th29}
	If hypotheses $H(\xi),H(\beta),H(f)'_2$ hold, then
	\begin{itemize}
		\item[(a)] for every $\lambda\geq\hat{\lambda}_1=\hat{\lambda}_1(p,\xi,\beta)$, problem \eqref{eql} has no positive solutions;
		\item[(b)] for every $\lambda<\hat{\lambda}_1$ problem \eqref{eql} has at least two positive solutions
		$$u_\lambda,\ \hat{u}_\lambda\in D_+,\ u_\lambda\leq\hat{u}_\lambda,\ u_\lambda\neq\hat{u}_\lambda;$$
		\item[(c)] for every $\lambda<\hat{\lambda}_1$ problem \eqref{eql} has a smallest positive solution $\bar{u}_\lambda\in D_+$ and the map $\lambda\mapsto\bar{u}_\lambda$ from $\mathcal{L}=(-\infty,\hat{\lambda}_1)$ into $C^1(\overline{\Omega})$ is
		\begin{itemize}
			\item[$\bullet$] strictly increasing (that is, $\vartheta<\lambda<\hat{\lambda}_1\Rightarrow\bar{u}_\lambda-\bar{u}_\vartheta\in {\rm int}\, C_+$);
			\item[$\bullet$] left continuous.
		\end{itemize}
	\end{itemize}
\end{theorem}

\medskip
{\bf Acknowledgements.} This research was supported by the Slovenian Research Agency Grants
P1-0292, J1-8131, J1-7025, N1-0064, and N1-0083. V.D.~R\u adulescu acknowledges the support through a grant of the Romanian Ministry of Research and Innovation, CNCS--UEFISCDI, project number PN-III-P4-ID-PCE-2016-0130,
within PNCDI III.

\end{document}